\documentclass{amsart}
\usepackage[ansinew]{inputenc}
\usepackage[english]{babel}
\usepackage{amsmath,latexsym,amssymb,verbatim}
\usepackage{pdfsync}

\usepackage[all]{xy}

\usepackage{hyperref}

\newtheorem{theorem}{Theorem}[section]
\newtheorem{lemma}[theorem]{Lemma}
\newtheorem{proposition}[theorem]{Proposition}
\newtheorem{corollary}[theorem]{Corollary}
\newtheorem{note}[theorem]{Note}

\newtheorem{Adjunction formula}[theorem]{Adjunction formula}
\newtheorem{remark}[theorem]{Remark}

\newtheorem{example}[theorem]{Example}
\newtheorem{examples}[theorem]{Examples}
\newtheorem{definition}[theorem]{Definition}
\newtheorem{notation}[theorem]{Notation}
\newtheorem{observation}[theorem]{Observation}

\DeclareMathOperator{\limi}{{lim}}
\newcommand{\ilim}[1]{\,\underset{#1}{\underset{\to}{\limi}}\,}
\newcommand{\plim}[1]{\,\underset{#1}{\underset{\leftarrow}{\limi}}\,}

\DeclareMathOperator{\Hom}{{Hom}}
\DeclareMathOperator{\Spec}{{Spec}}
\DeclareMathOperator{\Coker}{{Coker}}
\DeclareMathOperator{\Ker}{{Ker}}
\DeclareMathOperator{\Ima}{{Im}}

\newcommand{\dosflechasa}[3][]{\xymatrix@1{\ar@<1ex>[r]^-{#2}
\ar@<-1ex>[r]_-{#3} & }}

\NoCompileMatrices

\input arrow.tex

\begin{document}

\title{Reflexive functors in Algebraic Geometry}

%

\author{Pedro Sancho}
\address[Pedro Sancho]{Departamento de Matemáticas, Universidad de Extremadura,
Avenida de Elvas s/n, 06071 Badajoz, Spain}
\email{sancho@unex.es}
\thanks{Corresponding Author: Pedro Sancho.}

\date{April 12, 2012}

\begin{abstract} Reflexive functors of modules naturally appear in Algebraic Geometry. 
In this paper we define a wide and elementary family of reflexive functors of modules, closed by tensor products and  homomorphisms, in which Algebraic Geometry can be developed.
\end{abstract}

\maketitle

\section{Introduction}

Let $X$ be a  scheme over a field $K$. We can regard $X$ as a covariant functor of
sets from the category of commutative $K$-algebras to the category of sets through its functor of points $X^\cdot$, defined by $X^\cdot(S):=\Hom_{K-sch}(\Spec S,X)$, for all commutative $K$-algebras $S$. If $X=\Spec K[x_1,\ldots,x_n]/(p_1,\ldots,p_m)$ then
$$\aligned X^\cdot(S)& =\Hom_{K-sch}(\Spec S,X)=\Hom_{K-alg}(K[x_1,\ldots,x_n]/(p_1,\ldots,p_m),S)\\ & =\{s\in S^n\colon p_1(s)=\cdots=p_m(s)=0\}.\endaligned$$
It is well known that $\Hom_{K-sch}(X,Y)=\Hom_{funct.}(X^\cdot,Y^\cdot)$, and
$X$ is a group $K$-scheme if and only if $X^\cdot$ is a functor of groups.

We can regard $K$ as functor of rings $\mathcal K$,  by defining   $\mathcal K(S):=S$, for all commutative $K$-algebras $S$.
Let $V$ be a $K$-vector space. We can regard $V$ as a covariant functor of $\mathcal K$-modules, $\mathcal V$, by defining $\mathcal V(S):=V\otimes_K S$. We will say that $\mathcal V$ is the $\mathcal K$-quasi-coherent module associated with $V$.  If
$V=\oplus_IK$ then $\mathcal V(S)=\oplus_I S$. The category of $K$-vector spaces, ${\mathcal C}_{\text{K-vect}}$, is equivalent to the category of quasi-coherent $\mathcal K$-modules, ${\mathcal C}_{\text{qs-coh $\mathcal K$-mod}}$: the functors ${\mathcal C}_{\text{K-vect}} \rightsquigarrow {\mathcal C}_{\text{qs-coh $\mathcal K$-mod}}$, $V\rightsquigarrow \mathcal V$ and ${\mathcal C}_{\text{qs-coh $\mathcal K$-mod}}\rightsquigarrow{\mathcal C}_{\text{K-vect}},$ $\mathcal V\rightsquigarrow \mathcal V(K)$ give the equivalence.

It is well known that the theory of linear representations of a group scheme $G$ can be developed, via their associated functors, as a theory  of an abstract group and its linear representations. Thus, if $G=\Spec A$ is affine, the category of comodules over $A$ is equivalent to the category of quasi-coherent $G^\cdot$-modules.

Given a functor of $\mathcal K$-modules, $\mathbb M$ (that is, a covariant functor from the category of commutative $K$-algebras to the category of abelian groups, with a structure of $\mathcal K$-module), we denote $\mathbb M^*:=\mathbb Hom_{\mathcal K}(\mathbb M,\mathcal K)$ (see \ref{nota2.2}). We say that $\mathbb M$ is a reflexive functor of modules if $\mathbb M=\mathbb M^{**}$.

  If $V$ is an infinite dimensional vector space, $V\not\simeq V^{**}$. Moreover, if $G=\Spec A$ is a group scheme and $V$ is a comodule over $A$, $V^*$ is not a comodule over $A$ (as one would naively think). However, $\mathcal V=\mathcal V^{**}$ (see \cite{Amel} and \ref{reflex}) and $\mathcal V^*$ is a $G^\cdot$-module in the obvious way (see \cite{Amel3}). The category of comodules over $A$ is not equivalent to the category of $A^*$-modules, but the category of comodules over $A$ is equivalent to the category of quasi-coherent $\mathcal A^*$-modules.

 Let $\mathbb X$ be a functor of sets and $\mathbb A_{\mathbb X}:=\mathbb Hom_{funct.}(\mathbb X,\mathcal K)$ the functor of functions of $\mathbb X$. We say that $\mathbb X$ is an affine functor if  $\mathbb X=\Spec \mathbb A_{\mathbb X}:=\mathbb Hom_{\mathcal K-alg}(\mathbb A_{\mathbb X},\mathcal K)$ and $\mathbb A_{\mathbb X}$ is reflexive, see \ref{4.15} and \ref{4.16} for details (we warn the reader that  in the literature affine functors are sometimes defined to be functors of points of affine schemes). The functors of points of affine  schemes and formal schemes 
are affine functors  (see \ref{n3.22}). Let $\mathbb G$ be an affine functor of monoids. $\mathbb A_{\mathbb G}^*$ is a functor of algebras and the category of $\mathbb G$-modules is equivalent
to the category of $\mathbb A_{\mathbb G}^*$-modules. Applications of these results include Cartier duality (see \ref{Cartier}), neutral Tannakian duality for affine group schemes and the equivalence between formal groups and Lie algebras in characteristic zero (see \cite{navarro}).

In summary, functors from the category of commutative algebras to the category of sets (groups, rings, etc.) naturally appear in Algebraic Geometry and many definitions and results are better understood in this language. 
Many results are based in the reflexivity of the considered functors.

Every reflexive functor
is isomorphic to a  functor of $\mathcal K$-submodules  of $\prod_I\mathcal K$ (for some set $I$) and it is isomorphic to an inverse limit of $\mathcal K$-quasi-coherent modules, that is, it is pro-quasicoherent (see \cite{navarro}). Is this family closed under tensor products? 
 Is this family closed under homomorphisms?  We do not know to answer these questions without adding hypothesis. In this paper we define a wide and elementary family of reflexive functors of modules, closed by tensor products and  homomorphisms, in which Algebraic Geometry can be developed.

Now assume $K=R$ is a commutative ring. Let $I$ be a set and $\alpha\subseteq I$, we denote
$\mathcal R^{(\alpha}=\overset\alpha\oplus\mathcal R$ and we have the obvious inclusion
$\mathcal R^{(\alpha}\times \mathcal R^{I-\alpha}\subseteq \mathcal R^{\alpha}\times \mathcal R^{I-\alpha}=\mathcal R^I$. Let $H_\alpha:=\mathcal R^{(\alpha}\oplus \mathcal R^{I-\alpha}$. We say that a functor of 
$\mathcal R$-modules, $\mathbb M$,  is essentially free if there exist a subset $I$ and a subset  of the set of parts of $I$, $P$ such that
$$\mathbb M\simeq \underset{\alpha\in P} \cap H_\alpha$$
Let $P^\circ:=\{\beta\subseteq I\colon |\beta\cap \alpha|<\infty\}$. It is easy to check that
$\underset{\alpha\in P} \cap H_\alpha=\underset{\beta\in P^\circ} \cup \mathcal R^\beta$ (where $\mathcal R^\beta=\mathcal R^\beta\times\{0\}\subseteq \mathcal R^\beta\times \mathcal R^{I-\beta}=\mathcal R^I$).

Let $\mathfrak F$ be the category of essentially free $\mathcal R$-modules.
We prove (see \ref{3.6}):

 \begin{enumerate}

\item Essentially free functors of $\mathcal R$-modules are reflexive.

 \item If $M$ is a free $R$-module, then $\mathcal M$ is essentially free.

\item  Essentially free modules are pro-quasicoherent modules (see \ref{FP2}).

\item  If $\mathbb M,\mathbb M'\in\mathfrak F$, then $\mathbb Hom_{\mathcal R}(\mathbb M,\mathbb M')\in\mathfrak F$ and
$(\mathbb M\otimes_{\mathcal R}\mathbb M')^{**}\in\mathfrak F$, which satisfies
     $$
     \Hom_{\mathcal R}((\mathbb M\otimes_{\mathcal R}\mathbb M')^{**},\mathbb {M''})=\Hom_{\mathcal R}(\mathbb M\otimes_{\mathcal R}\mathbb M',\mathbb {M''}),$$
for every reflexive  functor of $\mathcal R$-modules, $\mathbb {M''}$.

\item If $\mathbb A,\mathbb A'\in\mathfrak F$ are functors of $\mathcal R$-algebras, then $(\mathbb A\otimes_{\mathcal R}\mathbb A')^{**}\in\mathfrak F$ is a functor of $\mathcal R$-algebras and
     $$
     \Hom_{\mathcal R-alg}((\mathbb A\otimes_{\mathcal R}\mathbb A')^{**},\mathbb {A''})=\Hom_{\mathcal R-alg}(\mathbb A\otimes_{\mathcal R}\mathbb A',\mathbb {A''}),$$
for every $\mathcal R$-algebra  $\mathbb {A''}$ which is a reflexive  functor of $\mathcal R$-modules.

\item If $\mathbb A,\mathbb B\in\mathfrak F$ are pro-quasicoherent functors of  algebras, then $(\mathbb A^*\otimes_{\mathcal R} \mathbb B^*)^*\in\mathfrak F$ and it is a functor of pro-quasicoherent   algebras, which satisfies
     $$\Hom_{\mathcal R-alg}((\mathbb A^*\otimes_{\mathcal R}\mathbb B^*)^*,\mathbb C)=
     \Hom_{\mathcal R-alg}(\mathbb A\otimes_{\mathcal R}\mathbb B,\mathbb C),$$
for every pro-quasicoherent functor of   algebras, $\mathbb C$ (see \ref{prodspec}).

\item If $\mathbb M,\mathbb M'\in\mathfrak F$, the natural morphism $\Hom_{\mathcal R}(\mathbb M,\mathbb M')\to \Hom_{R}(\mathbb M(R),\mathbb M'(R))$ is injective.
    If $\mathbb A\in \mathfrak F$ is a functor of $\mathcal R$-algebras and
    $\mathbb M,\mathbb M'\in\mathfrak F$ are functors of $\mathbb A$-modules, then a morphism
    of $\mathcal R$-modules $\mathbb M\to\mathbb M'$ is a morphism of $\mathbb A$-modules if and only if $\mathbb M(R)\to\mathbb M'(R)$ is a morphism of $\mathbb A(R)$-modules. Let $M$ be an $R$-module. If $\mathcal M$ is an $\mathbb A$-module, then the set of all quasi-coherent $\mathbb A$-submodules of $\mathcal M$ is equal to the set of all $\mathbb A(R)$-submodules of $M$ (see \ref{3.10X}, \ref{invqua4} and \ref{5.14}). These results have obvious applications to the theory of linear representations 
of  functors of monoids (see \ref{3.18}).    

\item[(7')] Assume that $K$ is a field and that $\mathfrak F$ is the family of reflexive functors of $\mathcal K$-modules. Then, all the results of $(7)$ are likewise true.
    
    \end{enumerate}

We prove that a functor of $\mathcal R$-modules $\mathbb C$ is  a functor of coalgebras if and only if
$\mathbb C^*$ is a functor of $\mathcal R$-algebras.
In particular, an $R$-module $C$ is a coalgebra if and only if $\mathcal C^*$ is a functor of $\mathcal R$-algebras.

We will say that a reflexive functor $\mathbb B$ of pro-quasicoherent algebras is a functor of bialegbras if  $\mathbb B^*$ is a functor of  algebras and the dual morphisms of the multiplication morphism $\mathbb B^*\otimes_{\mathcal R} \mathbb B^*\to \mathbb B^*$ and the unit morphism $\mathcal R\to \mathbb B^*$ are morphisms of functors of algebras (see \ref{bialgebras}). 

$B$ is an $R$-bialgebra (in the standard sense) if and only if $\mathcal B$ is a  functor of   bialgebras (see Proposition \ref{5.24}).

In the literature, there have been many attempts to obtain a well-behaved duality for non finite dimensional bialgebras (see \cite{Timmerman} and references
therein). One of them, for example, states that the functor that associates with  each bialgebra $A$ over a field $K$ the so-called dual bialgebra $A^\circ:=\ilim{I\in J} (A/I)^*$, where $J$ is the set of bilateral ideals $I\subset A$ such that $\dim_K A/I<\infty$,  is auto-adjoint (see \cite{E} and \ref{cite{E}}).
Another one associates with each bialgebra
$A$ over a pseudocompact ring $R$ the bialgebra $A^*$ endowed with a certain topology (see \cite[Exposé VII$_B$ 2.2.1]{demazure}).

In this paper (see \ref{dualbial}), we prove the following theorem.

\begin{theorem} Let ${\mathcal C}_{\mathfrak F-bialg}$ be the category of  pro-quasicoherent functors $\mathbb B\in\mathfrak F$  of  bialgebras.
The functor ${\mathcal C}_{\mathfrak F-bialg} \rightsquigarrow {\mathcal C}_{\mathfrak F-bialg}$, $ \mathbb B\rightsquigarrow \mathbb B^*$ is a categorical anti-equivalence.
\end{theorem}

Finally, let us comment some geometric aspects of this theory. Let $\mathbb A\in \mathfrak F$ be a functor of algebras. We prove that $\Spec \mathbb A$ is
a direct limit of affine schemes. 
Let $\mathbb X$ be a functor of sets and assume $\mathbb A_{\mathbb X}\in\mathfrak F$.  $\mathbb X$ is affine if and only if $\mathbb A_{\mathbb X}$ is a functor of  pro-quasicoherent algebras (see \ref{t1.92}). Classical formal schemes correspond to functors
$\Spec \mathcal C^*$ (where $C$ is a functor of algebras or equivalently $C$ is a coalgebra).
We prove that the category of affine functors of  monoids $\mathbb G$ is anti-equivalent to the category of functors of commutative bialgebras.
In particular, Cartier duality is obtained (see \ref{Cartier}).

\section{Preliminaries}

Let $R$ be a commutative ring (associative with a unit). All functors considered in this paper are covariant functors from the category of commutative $R$-algebras (always assumed to be associative with a unit) to the category of sets. A functor $\mathbb X$ is said to be a functor of sets (resp. groups, rings,  etc.) if $\mathbb X$ is a functor from the category of  commutative $R$-algebras to the category of sets (resp. groups, rings, etc.).

\begin{notation} \label{nota2.1}For simplicity, given a functor of sets $\mathbb X$,
we sometimes use $x \in \mathbb X$  to denote $x \in \mathbb X(S)$. Given $x \in \mathbb X(S)$ and a morphism of commutative $R$-algebras $S \to S'$, we still denote by $x$ its image by the morphism $\mathbb X(S) \to \mathbb X(S')$.\end{notation}

Let $\mathcal R$ be the functor of rings defined by ${\mathcal R}(S):=S$, for all commutative $R$-algebras $S$.
A functor of sets $\mathbb M$ is said to be a functor of $\mathcal R$-modules if we have morphisms of functors of sets, $\mathbb M\times \mathbb M\to \mathbb M$ and ${\mathcal R}\times \mathbb M\to \mathbb M$, so that
$\mathbb M(S)$ is an $S$-module, for every commutative $R$-algebra $S$.

 Let $\mathbb M$ and $\mathbb M'$ be functors of $\mathcal R$-modules.
 A morphism of functors of $\mathcal R$-modules $f\colon \mathbb M\to \mathbb M'$
 is a morphism of functors  such that the defined morphisms $f_S\colon \mathbb M(S)\to
 \mathbb M'(S)$ are morphisms of $S$-modules, for all commutative $R$-algebras $S$.
 We will denote by $\Hom_{\mathcal R}(\mathbb M,\mathbb M')$ the  family of all morphisms of $\mathcal R$-modules from $\mathbb M$ to $\mathbb M'$.

%

Given a commutative $R$-algebra $S$, we denote by $\mathbb M_{|S}$ the functor $\mathbb M$ restricted to the category of commutative $S$-algebras.
We will denote by ${\mathbb Hom}_{\mathcal R}(\mathbb M,\mathbb M')$\footnote{In this paper, we will only  consider well defined functors ${\mathbb Hom}_{\mathcal R}(\mathbb M,\mathbb M')$, that is to say, functors such that $\Hom_{\mathcal S}(\mathbb M_{|S},\mathbb {M'}_{|S})$ is a set, for all $S$.} the functor of $\mathcal R$-modules $${\mathbb Hom}_{\mathcal R}(\mathbb M,\mathbb M')(S):={\rm Hom}_{\mathcal S}(\mathbb M_{|S}, \mathbb M'_{|S}).$$  Obviously,
$$(\mathbb Hom_{\mathcal R}(\mathbb M,\mathbb M'))_{|S}=
\mathbb Hom_{\mathcal S}(\mathbb M_{|S},\mathbb M'_{|S}).$$

\begin{notation} \label{nota2.2} We denote $\mathbb M^*=\mathbb Hom_{\mathcal R}(\mathbb M,\mathcal R)$.\end{notation}


\begin{notation} Tensor products, direct limits, inverse limits, etc., of functors of $\mathcal R$-modules and  kernels, cokernels, images, etc.,  of morphisms of functors of $\mathcal R$-modules are regarded in the category of functors of $\mathcal R$-modules.\end{notation}

We have that
$$\aligned & (\mathbb M\otimes_{\mathcal R} \mathbb M')(S)=\mathbb M(S)\otimes_{S} \mathbb M'(S),\,
(\Ker f)(S)=\Ker f_S,\, (\Coker f)(S)=\Coker f_S,\\
& (\Ima f)(S)=\Ima f_S,\, (\ilim{i\in I} \mathbb M_i)(S)=\ilim{i\in I} (\mathbb M_i(S)),\,
(\plim{i\in I} \mathbb M_i)(S)=\plim{i\in I} (\mathbb M_i(S)).\endaligned $$

\begin{definition} Given an $R$-module $M$ (resp. $N$, etc.), ${\mathcal M}$  (resp. $\mathcal N$, etc.) will denote  the functor of $\mathcal R$-modules  defined by ${\mathcal M}(S) := M \otimes_R S$ (resp. $\mathcal N(S):=N\otimes_R S$, etc.). $\mathcal M$  will be called  quasi-coherent $\mathcal R$-module (associated with $M$).
\end{definition}

\begin{proposition} \cite[1.3]{Amel}\label{tercer}
For every functor of ${\mathcal R}$-modules $\mathbb M$ and every $R$-module $M$, it is satisfied that
$${\rm Hom}_{\mathcal R} ({\mathcal M}, \mathbb M) = {\rm Hom}_R (M, \mathbb M(R)).$$
\end{proposition}


The functors $M \rightsquigarrow {\mathcal M}$, ${\mathcal M} \rightsquigarrow {\mathcal M}(R)=M$ establish an equivalence between the category of $R$-modules and the category of quasi-coherent $\mathcal R$-modules (\cite[1.12]{Amel}). In particular, ${\rm Hom}_{\mathcal R} ({\mathcal M},{\mathcal M'}) = {\rm Hom}_R (M,M')$.
 For any pair of $R$-modules $M$ and $N$, the quasi-coherent module associated with $M\otimes_R N$ is $\mathcal M\otimes_{\mathcal R}\mathcal N$. ${\mathcal M}_{\mid S}$ is the quasi-coherent $\mathcal S$-module associated with $M \otimes_R
S$.

%
%
%
%

\begin{proposition}  \cite[1.8]{Amel}\label{prop4} \label{1.8Amel}
Let $M$ and $M'$ be $R$-modules. Then, $${\mathbb Hom}_{\mathcal R} ({\mathcal M^*}, {\mathcal M'}) = {\mathcal M} \otimes_{\mathcal R} {\mathcal M'}.$$
\end{proposition}

If $\mathcal M'=\mathcal R$, in the previous proposition, we obtain the following theorem.

\begin{theorem} \cite[1.10]{Amel}\label{reflex}
Let $M$ be an $R$-module. Then $${\mathcal M^{**}} = {\mathcal M}.$$
\end{theorem}

 A functor of $\mathcal R$-modules ${\mathcal M}^*$ is a quasi-coherent $\mathcal R$-module if and only if $M$ is a projective finitely generated  $R$-module (see \cite{Amel2}).

\begin{notation} Let $i\colon R\to S$ be a ring homomorphism between commutative rings.
Given a functor of $\mathcal R$-modules, $\mathbb M$, let $i^* \mathbb M$  be the functor of
$\mathcal S$-modules defined by
$(i^* \mathbb M)(S') := \mathbb M(S')$.
Given a functor of $\mathcal S$-modules, $\mathbb M'$, let $i_* \mathbb M'$ be the functor of
$\mathcal R$-modules defined by
$(i_* \mathbb M')(R') := \mathbb M(S \otimes_R R')$.
\end{notation}

\begin{Adjunction formula} \cite[1.12]{Amel} \label{adj}
Let $\mathbb M$ be a functor of ${\mathcal R}$-modules and let $\mathbb M'$ be a functor of  $\mathcal S$-modules. Then, 
$${\rm Hom}_{\mathcal S} (i^*\mathbb M, \mathbb M') = {\rm Hom}_{\mathcal R} (\mathbb M, i_*\mathbb M').$$
\end{Adjunction formula}

\begin{corollary} \label{adj2} Let $\mathbb M$ be a functor of $\mathcal R$-modules. Then
$$\mathbb M^*(S)=\Hom_{\mathcal R}(\mathbb M,\mathcal S),$$
for all commutative $R$-algebras $S$.\end{corollary}

\begin{proof} $\mathbb M^*(S)=\Hom_{\mathcal S}(\mathbb M_{|S},\mathcal S)
\overset{\text{\ref{adj}}}=\Hom_{\mathcal R}(\mathbb M,\mathcal S)$.
\end{proof}

\begin{definition} Let $\mathbb M$ be a functor of $\mathcal R$-modules. We will say that
$\mathbb M^*$ is a dual functor.
We will say that a functor of $\mathcal R$-modules ${\mathbb M}$ is reflexive if ${\mathbb M}={\mathbb M}^{**}$.\end{definition}

\begin{examples}  Quasi-coherent modules and module schemes are reflexive functors of $\mathcal R$-modules.\end{examples}

\begin{proposition} \label{3.2}
Let ${\mathbb M}$ be a functor of $\mathcal R$-modules such that ${\mathbb M}^*$ is a reflexive
functor. The closure of dual functors of $\mathcal R$-modules of ${\mathbb M}$ is
${\mathbb M}^{**}$, that is, we have the functorial equality $${\rm
Hom}_\mathcal R({\mathbb M},{\mathbb M}')={\rm Hom}_\mathcal R({\mathbb M}^{**},{\mathbb M'}),$$ for every dual functor of
$\mathcal R$-modules ${\mathbb M}'$.
\end{proposition}

\begin{proof} Write $\mathbb M'=\mathbb {N}^*$. Then,
${\rm Hom}_{\mathcal R}({\mathbb M},{\mathbb M}')\! =\!{\rm Hom}_{\mathcal R}({\mathbb M}\otimes{\mathbb N},\mathcal R)\!=\!{\rm Hom}_{\mathcal R}({\mathbb N},{\mathbb M}^*)={\rm
Hom}_{\mathcal R}({\mathbb N}\otimes{\mathbb M}^{**},\mathcal R) ={\rm Hom}_{\mathcal R}({\mathbb M}^{**},{\mathbb M}').$
\end{proof}

A functor of rings (associative with a unit), $\mathbb A$,
is said to be a functor of $\mathcal R$-algebras if we have a morphism of functors of rings
$\mathcal R\to \mathbb A$
(and $\mathcal R(S)=S$ commutes with all elements of $\mathbb A(S)$, for every commutative $R$-algebra $S$).

\begin{proposition}\label{2.4}
Let $\mathbb A$ be a functor of $\mathcal R$-algebras such that $\mathbb A^*$ is a
reflexive functor of $\mathcal R$-modules. The closure of dual functors of
$\mathcal R$-algebras of $\mathbb A$ is $\mathbb A^{**}$, that is, we have the functorial
equality $${\rm Hom}_{\mathcal R-alg}(\mathbb A,\mathbb B)={\rm Hom}_{\mathcal R-alg}(\mathbb A^{**},\mathbb B),$$
for every functor of $\mathcal R$-algebras $\mathbb B$, such that $\mathbb B$ is a dual functor of $\mathcal R$-modules.

As a consequence, the category of dual functors of ${\mathbb A}$-modules is equal to the category of dual functors of ${\mathbb A}^{**}$-modules.
\end{proposition}

\begin{proof}
 Given a dual functor of ${\mathcal R}$-modules $\mathbb M^*$,
 $$\aligned {\mathbb Hom}_{\mathcal R} ({\mathbb A} \otimes \overset n\ldots \otimes {\mathbb A}, \mathbb M^*) &
={\mathbb Hom}_{\mathcal R} ({\mathbb A} \otimes \overset{n-1}\ldots \otimes {\mathbb A}, \mathbb Hom_{\mathcal R}(\mathbb A,\mathbb M^*))
\\ & \overset{\text{\ref{3.2}}} =
{\mathbb Hom}_{\mathcal R} ({\mathbb A} \otimes \overset{n-1}\ldots \otimes {\mathbb A}, \mathbb Hom_{\mathcal R}(\mathbb A^{**},\mathbb M^*))\\
& \overset{\text{Ind.Hyp.}}= {\mathbb Hom}_{\mathcal R} ({\mathbb A}^{**} \otimes \overset{n-1}\ldots \otimes {\mathbb A}^{**}, \mathbb Hom_{\mathcal R}(\mathbb A^{**},\mathbb M^*))
\\ & = {\mathbb
Hom}_{\mathcal R} ({\mathbb A}^{**} \otimes \overset n\ldots \otimes {\mathbb A}^{**}, \mathbb M^*),\endaligned$$
by induction on $n$.
Let $i\colon \mathbb A\to \mathbb A^{**}$ be the natural morphism.
The multiplication morphism $m\colon  \mathbb A\otimes  \mathbb A\to  \mathbb A$
defines a unique morphism
$m'\colon  \mathbb A^{**}\otimes  \mathbb A^{**}\to  \mathbb A^{**}$  such that the diagram
$$\xymatrix{  \mathbb A \otimes  \mathbb A \ar[d]^-m \ar[r]^-{i\otimes i} &  \mathbb A^{**}\otimes  \mathbb A^{**}\ar[d]^-{m'}\\  \mathbb A \ar[r]^-i &  \mathbb A^{**}}$$
is commutative, because $\Hom_{\mathcal R}(\mathbb A\otimes \mathbb A,\mathbb A^{**})
=\Hom_{\mathcal R}(\mathbb A^{**}\otimes \mathbb A^{**},\mathbb A^{**})$.
It follows easily that the algebra structure of
${\mathbb A}$ defines an algebra structure on ${\mathbb A}^{**}$.
Let us only check that $m'$ satisfies the associative property: The morphisms
$m'\circ (m'\otimes {\rm Id})$, $m'\circ ({\rm Id}\otimes m')\colon
\mathbb A^{**}\otimes\mathbb A^{**}\otimes \mathbb A^{**}\to \mathbb A^{**}$ are
equal because
$$\aligned & (m'\circ (m'\otimes {\rm Id}))\circ (i\otimes i\otimes i) =
m'\circ (i\otimes i)\circ (m\otimes {\rm Id}) =
i\circ m\circ (m\otimes {\rm Id})\\ & =i\circ m\circ ({\rm Id}\otimes m)=
m'\circ (i\otimes i)\circ ({\rm Id}\otimes m)=
(m'\circ ({\rm Id}\otimes m'))\circ (i\otimes i\otimes i).\endaligned$$

 The kernel of the morphism
$$\Hom_{\mathcal R}(\mathbb A,\mathbb B)\to \Hom_{\mathcal R}(\mathbb A\otimes_{\mathcal R}\mathbb A,\mathbb B),\,\,f\mapsto f\circ m-m\circ (f\otimes f),$$ coincides with the
kernel of the morphism
$$\Hom_{\mathcal R}(\mathbb A^{**},\mathbb B)\to \Hom_{\mathcal R}(\mathbb A^{**}\otimes_{\mathcal R}\mathbb A^{**},\mathbb B),\,\,f\mapsto f\circ m'-m\circ (f\otimes f).$$
Then, ${\rm Hom}_{{\mathcal R}-alg}({\mathbb A},{\mathbb B})={\rm
Hom}_{{\mathcal R}-alg}({\mathbb A}^{**},{\mathbb B})$.

Finally, given a dual functor of $\mathcal R$-modules $\mathbb M^*$, then $\mathbb End_{\mathcal R}\mathbb M^*=(\mathbb M^*\otimes\mathbb M)^*$ is a dual functor of $\mathcal R$-modules and
$${\rm Hom}_{{\mathcal R}-alg}({\mathbb A},\mathbb End_{\mathcal R}\mathbb M^*)={\rm
Hom}_{{\mathcal R}-alg}({\mathbb A}^{**},\mathbb End_{\mathcal R}\mathbb M^*).$$
Given two $\mathbb
A^{**}$-modules, $\mathbb M'$ and $\mathbb M^*$, and a morphism $f\colon \mathbb M'\to \mathbb M^*$ of $\mathbb A$-modules, then $f$ is a morphism of ${\mathbb A}^{**}$-modules
because given $m'\in\mathbb M'$, the morphism $\mathbb A^{**}\to \mathbb M^*$, $a\mapsto
f(am')-af(m')$ is zero because $\mathbb A\to \mathbb M^*$, $a\mapsto
f(am')-af(m')$ is zero. Then, the category of dual functors of ${\mathbb A}$-modules is equal to the category of  dual functors of ${\mathbb A}^{**}$-modules.

\end{proof}

\begin{notation} Given a functor of sets $\mathbb X$, the functor $\mathbb A_{\mathbb X}:={\mathbb Hom}({\mathbb X}, {\mathcal
R})$ is said to be the functor  of functions of $\mathbb X$.\end{notation}

\begin{example} Let $X=\Spec A$ be an affine $R$-scheme and let $X^\cdot$ be its functor of points, that is,
$$X^\cdot(S):=\Hom_{R-sch}(\Spec S,X)=\Hom_{R-alg}(A,S).$$
Then, $\mathbb A_{X^\cdot}={\mathbb Hom}({X^\cdot }, {\mathcal
R})=\mathbb Hom_{\mathcal R-alg}(\mathcal R[x],\mathcal A)=\mathcal A$.\end{example}

\begin{definition} Let $\mathbb X$ be a functor of sets.
 Let $\mathcal R\mathbb X$ be the functor of $\mathcal R$-modules defined by $$\mathcal R\mathbb X(S):=\oplus_{\mathbb X(S)} S=\{ \mbox{formal finite $S$-linear  combinations of elements of } \mathbb X(S) \}$$\end{definition}

 Obviously, ${\mathbb Hom}(\mathbb X,{\mathbb M})=
{\mathbb Hom}_{\mathcal R}(\mathcal R\mathbb X,{\mathbb M})$, for all functors of $\mathcal R$-modules, $\mathbb M$.

Observe that $\mathbb A_{\mathbb X}={\mathbb Hom}(\mathbb X,{\mathcal R})=(\mathcal R\mathbb X)^*$.

\begin{proposition}\label{2.3} Let $\mathbb X$ be a functor of sets and assume $\mathbb A_{\mathbb X}=\mathbb B^*$. Then,
$${\rm Hom}({\mathbb X}, {\mathbb M^*}) =  {\rm Hom}_{\mathcal R}(\mathbb B, {\mathbb M^*}),$$ for every dual functor of $\mathcal R$-modules ${\mathbb M}^*$. In particular, if $\mathbb A_{\mathbb X}$ is reflexive, then 
$${\rm Hom}({\mathbb X}, {\mathbb M^*}) =  {\rm Hom}_{\mathcal R}(\mathbb A_{\mathbb X}^*, {\mathbb M^*}),$$
\end{proposition}

\begin{proof} We have that
 $$\aligned {\rm Hom}({\mathbb X}, {\mathbb M}^*) & =  {\rm Hom}_{\mathcal R}(\mathcal R\mathbb X, {\mathbb M}^*) ={\rm Hom}_{\mathcal R}(\mathcal R\mathbb X\otimes_{\mathcal R}{\mathbb M},\mathcal R) ={\rm Hom}_{\mathcal R}({\mathbb M}, {\mathbb A}_{\mathbb X})\\ & ={\rm Hom}_{\mathcal R}(\mathbb B\otimes_{\mathcal R} {\mathbb M},\mathcal R)={\rm Hom}_{\mathcal R}(\mathbb B, {\mathbb M}^*).\endaligned$$

\end{proof}

Let $\mathbb G$ be a functor  of monoids. ${\mathcal R}\mathbb G$ is obviously a functor of
${\mathcal R}$-algebras. Given a functor of ${\mathcal R}$-algebras $\mathbb B$, it is
easy to check the equality $${\rm Hom}_{mon} (\mathbb G, {\mathbb B})
= {\rm Hom}_{{\mathcal R}-alg} ({\mathcal R}\mathbb G, \mathbb B).$$

\begin{theorem}\label{2.5}
Let $\mathbb G$ be a functor of monoids with a reflexive functor of functions. Then, the
closure of dual functors of algebras of $\mathbb G$ is
$\mathbb A_{\mathbb G}^*$. That is,

$${\rm Hom}_{mon}(\mathbb G, {\mathbb B}) = {\rm Hom}_{\mathcal R-alg}(\mathcal R\mathbb G,
{\mathbb B}) = {\rm Hom}_{\mathcal R-alg}({\mathbb A_{\mathbb G}^*}, {\mathbb B}),$$
for every dual functor of $\mathcal R$-algebras $\mathbb B$.

The category of quasi-coherent ${\mathbb G}$-modules is equi\-valent to the
category of quasi-coherent $\mathbb A_{\mathbb G}^*$-modules.
Likewise, the category of
dual functors of $\mathbb G$-modules is equivalent to the category of dual
functors of $\mathbb A_{\mathbb G}^*$-modules.
\end{theorem}

\begin{proof}
$(\mathcal R\mathbb G)^*=\mathbb A_{\mathbb G}$ is reflexive.
By Proposition \ref{2.4}, it is easy to complete the proof.
\end{proof}

Remark that the structure of functor of algebras of
$\mathbb A_{\mathbb G}^*$ is the only one that makes the morphism ${\mathbb G}\to \mathbb A_{\mathbb G}^*$ a morphism of functors of
monoids.

\begin{example} \label{E2.20} Let $G=\Spec A$ be an $R$-group scheme. The category of $G$-modules (that is to say,  the category of
comodules over $A$)
is equivalent to the category of $\mathcal R$-quasi-coherent $\mathcal R[G^\cdot]$-modules. Then, the category of $G$-modules is equivalent to the category of  $\mathcal R$-quasi-coherent $\mathcal A^*$-modules.

\end{example}

\section{Essentially free $\mathcal R$-modules}

\label{F}

\begin{notation} Let $P$ be a subset of the set of parts of a set $I$.  Denote $P^\circ:=\{\beta\subseteq I\colon \beta\cap\alpha \text{ is a finite set for all } \alpha\in P\}$.
\end{notation}

 If $P\subseteq P_2$ then $P_2^\circ\subseteq P^\circ$.
Obviously, $P\subseteq P^{\circ\circ}$. Then, $ P^{\circ} \subseteq (P^{\circ})^{\circ\circ}=(P^{\circ\circ})^\circ\subseteq
 P^{\circ}$. Hence, $$P^{\circ}=P^{\circ\circ\circ}.$$

Given a set $\alpha$, we denote $\mathcal R^\alpha=\prod_\alpha \mathcal R$ and  $\mathcal R^{(\alpha}=\oplus_\alpha \mathcal R$.
If $\alpha_1\subseteq \alpha_2$, we have the obvious epimorphism $\mathcal R^{(\alpha_2}
\to \mathcal R^{(\alpha_1}$, $(\lambda_i)_{i\in \alpha_2}\mapsto (\lambda_i)_{i\in \alpha_2}$ and the obvious injective morphism $\mathcal R^{\alpha_1}
\to \mathcal R^{\alpha_2}$, $(\lambda_i)_{i\in \alpha_1}\mapsto (\lambda_i)_{i\in \alpha_1}:=(\mu_i)_{i\in \alpha_2}$, where $\mu_i=\lambda_i$ if $i\in\alpha_1$ and $\mu_i=0$ if $i\notin\alpha_1$.

\begin{theorem} \label{teoremon}  Let $P$ be a subset of the set of parts of $I$, such that if
$\alpha,\alpha'\in P$ then $\alpha\cup\alpha'\in P$  and assume, for simplicity,  $\cup_{\alpha\in P} \alpha=I$. Then,

\begin{enumerate}
\item $\plim{\alpha\in P} \mathcal R^{(\alpha}=\ilim{\beta\in P^\circ} \mathcal R^\beta\subseteq \prod_I\mathcal R$. $\plim{\alpha\in P} \mathcal R^{(\alpha}=\underset{\alpha\in P}\cap
\mathcal R^{(\alpha} \times \mathcal R^{I-\alpha}$.

\item  $(\plim{\alpha\in P} \mathcal R^{(\alpha})^*=\plim{\beta\in P^\circ} \mathcal R^{(\beta}$ and $\plim{\alpha\in P} \mathcal R^{(\alpha}$ is reflexive. Then, 

$$\plim{\alpha\in P} \mathcal R^{(\alpha}=\plim{\alpha\in P^{\circ\circ} }\mathcal R^{(\alpha}.$$ 

\end{enumerate}
\end{theorem}

\begin{proof} 
1. It is easy to prove that $\plim{\alpha\in P} \mathcal R^{(\alpha}=\{(\lambda_i)\in \mathcal R^I\colon \text{ for each } \alpha\in P,\, \lambda_i=0 \text{ forall } i\in\alpha \text{ but a finite number of } i \}$. Let $m=(\lambda_i)\in\mathcal R^I$ and $\beta=\{i\in I\colon \lambda_i\neq 0\}$. Then, $m\in \plim{\alpha\in P} \mathcal R^{(\alpha}$ if and only if $|\beta\cap\alpha|<\infty$, for all $\alpha\in P$. Hence,  $\plim{\alpha\in P} \mathcal R^{(\alpha}=\ilim{\beta\in P^\circ} \mathcal R^\beta$.

$\plim{\alpha\in P} \mathcal R^{(\alpha}=\{(\lambda_i)\in \mathcal R^I\colon \text{for each } \alpha\in P,\, \lambda_i=0\, \forall\, i\in\alpha \text{ but a finite number of } i \}=\underset{\alpha\in P}\cap \{(\lambda_i)\in \mathcal R^I\colon \lambda_i=0 \text{ forall } i\in\alpha \text{ but a finite number of } i \} =
\underset{\alpha\in P}\cap \mathcal R^{(\alpha} \times \mathcal R^{I-\alpha}$.

2. $(\plim{\alpha\in P} \mathcal R^{(\alpha})^*=(\ilim{\beta\in P^\circ} \mathcal R^\beta)^*=\plim{\beta\in P^\circ}( \mathcal R^\beta)^*=\plim{\beta\in P^\circ} \mathcal R^{(\beta}$. Now,

$$(\plim{\alpha\in P} \mathcal R^{(\alpha})^{**}=\plim{\alpha\in P^{\circ\circ}} \mathcal R^{(\alpha}=\ilim{\alpha\in P^{\circ\circ\circ}} \mathcal R^{\alpha}=
\ilim{\alpha\in P^{\circ}} \mathcal R^{\alpha}=\plim{\alpha\in P} \mathcal R^{(\alpha}.$$
\end{proof}

\begin{observation} Let $P$ be the set of all numerable subsets of $I$. Then $P^\circ$ is the set of all finite subsets of $I$ and $P^{\circ\circ}$ is the set of all subsets of $I$. Hence, $(\ilim{\alpha\in P} \mathcal R^{\alpha})^{**}=
(\plim{\alpha\in P} \mathcal R^{(\alpha})^*=\plim{\alpha\in P^\circ} \mathcal R^{(\alpha}=\ilim{\alpha\in P^{\circ\circ}} \mathcal R^{\alpha}=\prod_I\mathcal R$.
Then, if $I$ is no numerable, $\ilim{\alpha\in P} \mathcal R^{\alpha}\underset\neq\subset (\ilim{\alpha\in P} \mathcal R^{\alpha})^{**}$. Moreover, 
$(\plim{\alpha\in P} \mathcal R^{(\alpha})^*\neq \ilim{\alpha\in P} ( \mathcal R^{(\alpha})^*$, because the latter term is not a reflexive functor of modules.

\end{observation}

\begin{theorem}  \label{teoremonb} 
Let $L$ be a set of indices. For each $l\in L$, let $P_l$ be a subset of the set of parts of a set $I_l$, such that if
$\alpha,\alpha'\in P_l$ then $\alpha\cup\alpha'\in P_l$. Let $Q:=\{(\alpha_l)_{l\in L}\in \prod_{l\in L} P_l\colon
\alpha_l=\emptyset$ for all $l\in L$ except for a finite number of $l\in L\}$.
Consider $\alpha=(\alpha_l)_{l\in L}$ as a subset of $\coprod_{l\in L} I_l$ as follows $\alpha=\coprod_{l\in L} \alpha_l$
Denote
$\mathbb M_l:=\plim{\alpha_l\in P_l} \mathcal R^{(\alpha}$ for each $l\in L$. Then, $$\prod_{l\in L} \mathbb M_l=\plim{\alpha\in Q} \mathcal R^{(\alpha}\,\text{ and }\, 
(\prod_{l\in L} \mathbb M_l)^*=\oplus_{l\in L} \mathbb M_l^*.$$
\end{theorem}

\begin{proof} 
 $\prod_{l\in L} \mathbb M_l=\plim{J\subset L,|J|<\infty} \prod_{j\in J} \mathbb M_j=
\plim{J\subset L,|J|<\infty} \plim{\alpha\in \prod_{j\in J}P_j} \mathcal R^{(\alpha}
=\plim{\alpha\in Q} \mathcal R^{(\alpha}$.

For the last statement, we can assume $P_l=P_l^{\circ\circ}$, by \ref{teoremon} 2.
Observe that $Q^\circ = \prod_{l\in L} P_l^\circ$ and that $Q^{\circ\circ} = Q$. Then, 
$$(\prod_{l\in L} \mathbb M_l)^*=\plim{\alpha\in Q^\circ} \mathcal R^{(\alpha}\underset{Q=Q^{\circ\circ}}=
\ilim{\alpha\in Q} \mathcal R^{\alpha}=\ilim{\alpha\in Q}(\oplus_{l\in L}\mathcal R^{\alpha_l})=
\oplus_{l\in L} \mathbb M_l^*.$$
\end{proof}

\begin{theorem}  \label{teoremonc} 

Let $P$ be a subset of the set of parts of $I$, such that if
$\alpha,\alpha'\in P$ then $\alpha\cup\alpha'\in P$  and assume, for simplicity,  $\cup_{\alpha\in P} \alpha=I$. Let $Q$ be a subset of the set of parts of $J$, such that if
$\beta,\beta'\in Q$ then $\beta\cup\beta'\in Q$  and assume, for simplicity,  $\cup_{\beta\in Q} \alpha=J$. Then, \begin{enumerate}

\item $\mathbb Hom_{\mathcal R}(\plim{\alpha\in P} \mathcal R^{(\alpha},\plim{\beta\in Q} \mathcal R^{(\beta})=\plim{(\alpha,\beta)\in P^\circ\times Q} \mathcal R^{(\alpha\times\beta}$.

\item  $(\plim{\alpha\in P} \mathcal R^{(\alpha}\otimes_{\mathcal R} \plim{\beta\in Q} \mathcal R^{(\beta})^{*}=\plim{\gamma\in P^\circ\times Q^\circ} \mathcal R^{(\gamma}$.
Hence,  \begin{enumerate}

\item $(\plim{\alpha\in P} \mathcal R^{(\alpha}\otimes_{\mathcal R} \plim{\beta\in Q} \mathcal R^{(\beta})^{**}=\plim{\gamma\in (P^\circ\times Q^\circ)^\circ} \mathcal R^{(\gamma}$.

\item     $((\plim{\alpha\in P} \mathcal R^{(\alpha})^*\otimes_{\mathcal R} (\plim{\beta\in Q} \mathcal R^{(\beta})^*)^{*}=\plim{\gamma\in P^{\circ\circ}\times Q^{\circ\circ}} \mathcal R^{(\gamma}$. \end{enumerate}

\end{enumerate} \end{theorem}

\begin{proof}

 $$\aligned 1.\,\, \mathbb Hom_{\mathcal R}(\plim{\alpha\in P} \mathcal R^{(\alpha}&,\plim{\beta\in Q} \mathcal R^{(\beta})  =\mathbb Hom_{\mathcal R}(\ilim{\alpha\in P^\circ} \mathcal R^{\alpha},\plim{\beta\in Q} \mathcal R^{(\beta})\\ &=\plim{(\alpha,\beta)\in P^\circ\times Q} \mathbb Hom_{\mathcal R}(\mathcal R^{\alpha},\mathcal R^{(\beta})=
\plim{(\alpha,\beta)\in P^\circ\times Q} \mathcal R^{(\alpha}\otimes_{\mathcal R} \mathcal R^{(\beta}\\ &=
\plim{(\alpha,\beta)\in P^\circ\times Q} \mathcal R^{(\alpha\times\beta}.\endaligned$$

$$\aligned \!\!\!\!2.\,\, (\plim{\alpha\in P} \mathcal R^{(\alpha}\otimes_{\mathcal R} \plim{\beta\in Q} \mathcal R^{(\beta})^{*} & =\mathbb Hom_{\mathcal R}(\plim{\alpha\in P} \mathcal R^{(\alpha}\otimes_{\mathcal R}\plim{\beta\in Q} \mathcal R^{(\beta},\mathcal R)\\ &=
\mathbb Hom_{\mathcal R}(\plim{\alpha\in P} \mathcal R^{(\alpha},\plim{\beta\in Q^\circ} \mathcal R^{(\beta}) =\plim{\gamma\in P^\circ\times Q^\circ} \mathcal R^{(\gamma}.\endaligned$$

\end{proof}

\begin{definition} \label{defi3.4}We will say that a functor of  $\mathcal R$-modules, $\mathbb M$, is essentially free if there exist a set $I$, a subset $P$ of the set of parts of $I$ (such that if
$\alpha,\alpha'\in P$ then $\alpha\cup\alpha'\in P$) and an isomorphism of functors of  $\mathcal R$-modules
$$\mathbb M\simeq \plim{\alpha\in P} \mathcal R^{(\alpha}.$$
Let $\mathfrak F$ be the category of essentially free $\mathcal R$-modules.
\end{definition}

\begin{examples} \label{Examples5.4}  If $V$ is a free $\mathcal R$-module, $\mathcal V$ and $\mathcal V^*\in\mathfrak F$.

\end{examples}

\begin{theorem}  \label{3.6} \begin{enumerate} 
\item Essentially free modules are reflexive.

\item Let $\mathbb M_l\in \mathfrak F$, for all $l\in L$. Then, $\prod_{l\in L} \mathbb M_L$ and $\oplus_{l\in L} \mathbb M_l\in \mathfrak F$.

\item If $\mathbb M,\mathbb M'\in \mathfrak F$, then $\mathbb M^*,\mathbb Hom_{\mathcal R}(\mathbb M,\mathbb M')\in \mathfrak F$.

\item If $\mathbb M,\mathbb M'\in \mathfrak F$, then $(\mathbb M\otimes_{\mathcal R}\mathbb M')^{*}\in\mathfrak F$ and $(\mathbb M\otimes_{\mathcal R}\mathbb M')^{**}$ satisfies
 $$
     \Hom_{\mathcal R}((\mathbb M\otimes_{\mathcal R}\mathbb M')^{**},\mathbb {M''})\overset{\text{\ref{3.2}}}=\Hom_{\mathcal R}(\mathbb M\otimes_{\mathcal R}\mathbb M',\mathbb {M''}),$$
for every dual  functor of $\mathcal R$-modules, $\mathbb {M''}$.

\item  If $\mathbb A,\mathbb A'\in\mathfrak F$ are functor of $\mathcal R$-algebras, then $(\mathbb A\otimes_{\mathcal R}\mathbb A')^{**}\in\mathfrak F$ is a functor of $\mathcal R$-algebras and
     $$
     \Hom_{\mathcal R-alg}((\mathbb A\otimes_{\mathcal R}\mathbb A')^{**},\mathbb {A''})\overset{\text{\ref{2.4}}}=\Hom_{\mathcal R-alg}(\mathbb A\otimes_{\mathcal R}\mathbb A',\mathbb {A''}),$$
for every $\mathcal R$-algebra  $\mathbb {A''}$ such that it is a dual  functor of $\mathcal R$-modules.

\end{enumerate}

\end{theorem}

\begin{proof} It is consequence of \ref{teoremon}, \ref{teoremonb} and \ref{teoremonc}.\end{proof}

\begin{lemma} \label{x} Let $\phi\colon \mathcal R^\alpha\to \mathcal M$ be a morphism of $\mathcal R$-modules. Then, $\Ima \phi_R\subseteq M$ is a $R$-module of finite type and $\phi$
 factors uniquely through an epimorphism onto the quasi-coherent module associated with  $\Ima \phi_R$.
\end{lemma}

\begin{proof} $\Hom_{\mathcal R}(\mathcal R^\alpha,\mathcal M)=R^{(\alpha}\otimes_R M$. Let $\{1_i\}_{i\in\alpha}$ be the standard basis of $R^{(\alpha}$. Then, $\phi=1_{i_1}\otimes m_1+\cdots+1_{i_n}\otimes m_n$, $N:=\Ima\phi_R=\langle m_1,\ldots,m_n\rangle$ and $\phi$
 factors uniquely through $1_{i_1}\otimes m_1+\cdots+1_{i_n}\otimes m_n\in \Hom_{\mathcal R}(\mathcal R^\alpha,\mathcal N)=R^{(\alpha}\otimes_R N$.

\end{proof}

\begin{proposition} \label{proposition3.12} $\mathcal M\in  \mathfrak F$  if and only if
$M$ is a free $R$-module. \end{proposition}

\begin{proof} If $\mathcal M\in  \mathfrak F$ then $\mathcal M=\ilim{\alpha\in P}\mathcal R^\alpha$. Consider the obvious injective morphism $i\colon \mathcal R^\alpha\to \mathcal M$. By Lemma \ref{x}, $\Ima i_R$ is a $R$-module of finite type, then $\alpha$ is a finite set. Hence, $\mathcal M=\ilim{\alpha\in P}\mathcal R^\alpha=\mathcal R^{(\cup_{\alpha \in P}\alpha}$ and $M$ is a free $R$-module.

\end{proof}

\begin{lemma} \label{invqua} Let $\mathbb M\in\mathfrak F$ and let $M$ be an $R$-module. Then, every morphism of $\mathcal R$-modules $\phi\colon \mathbb M\to \mathcal M$ factors uniquely through an epimorphism onto the quasi-coherent module associated with the $R$-submodule $\Ima \phi_R\subseteq M$.\end{lemma}

\begin{proof} It is a consequence of Lemma \ref{x} and the equality  $\mathbb M=\ilim{\alpha\in P} \mathcal R^\alpha$.\end{proof}

\begin{theorem} \label{FP2} Let $\{\mathcal M_i\}_{i\in I}$ be the set of all
quasi-coherent quotients of $\mathbb M\in\mathfrak F$. Then,
$\mathbb M^*=\ilim{i\in I}  \mathcal M_i^*$.
Therefore, $$\mathbb M=\plim{i\in I}  \mathcal M_i.$$

\end{theorem}

\begin{proof} Let $S$ be a commutative $R$-algebra.
$\mathbb M^*(S)=\Hom_{\mathcal R}(\mathbb M,\mathcal S)$, by Corollary \ref{adj2}.
The morphism $\ilim{i\in I}  \mathcal M_i^*(S)\to \Hom_{\mathcal R}(\mathbb M,\mathcal S)=\mathbb M^*(S)$ is obviously injective, and it is surjective by  Lemma \ref{invqua}. Hence, $\mathbb M^*=\ilim{i\in I}  \mathcal M_i^*$ and $\mathbb M=\mathbb M^{**}=\plim{i\in I}  \mathcal M_i.$\end{proof}

\begin{remark} \label{Remark} Let $R=K$ be a field. Now assume $\mathfrak F$ is the family of reflexive functors of $\mathcal K$-modules. Lemma \ref{invqua} is true by \cite[2.13]{Amel}. Hence,  Theorem \ref{FP2} is likewise true. Hence, every reflexive module is a direct limit of modules $\mathcal M_i^*$ and it is an inverse limit of quasicoherent modules

\end{remark}

\section{Homomorphisms and applications}

\begin{proposition} \label{3.10X}  Let  $\mathbb M\in\mathfrak F$ and let $\mathbb M'$ be  a dual functor of $\mathcal R$-modules. Then, the morphism
$$\Hom_{\mathcal R}(\mathbb M,\mathbb M')\to \Hom_{R}(\mathbb M(R),\mathbb M'(R)),\,\, \phi\mapsto\phi_R$$
is injective..\end{proposition}

\begin{proof} $\mathbb M=\ilim{\alpha\in P}\mathcal R^\alpha$. It is sufficient to prove this proposition when $\mathbb M=\mathcal R^\alpha$. Now,

$$\Hom_{\mathcal R}(\mathcal  R^\alpha,\mathbb N^*)=
\Hom_{\mathcal R}(\mathbb N, \mathcal R^{(\alpha})\subseteq \Hom_{\mathcal R}(\mathbb N, \mathcal R^{\alpha})=\prod^\alpha \mathbb N^*(R).$$
Hence, the composite morphism $\Hom_{\mathcal R}(\mathcal R^\alpha,\mathbb M')
\to \Hom_{R}( R^\alpha,\mathbb M'(R))\to \prod^\alpha \mathbb M'(R)$ is injective. Then, the morphism $\Hom_{\mathcal R}(\mathcal R^\alpha,\mathbb M')
\to \Hom_{R}(R^\alpha,\mathbb M'(R))$ is injective.

\end{proof}

\begin{proposition} \label{5.14} Let $\mathbb A\in\mathfrak F$ be a functor of  $\mathcal R$-algebras, let $\mathbb M\in\mathfrak F$  be a functor of $\mathbb A$-modules and let  $\mathbb M'$ be  a functor of $\mathbb A$-modules and a dual functor of $\mathcal R$-modules.
Let $f\colon \mathbb M\to \mathbb M'$ be a morphism of $\mathcal R$-modules. 
Then, $f$ is a morphism of $\mathbb A$-modules if and only if $f_R$  is a morphism of $\mathbb A(R)$-modules.

\end{proposition}

\begin{proof} The morphism $f$ is a morphism of $\mathbb A$-modules if and only if $F\colon \mathbb A\otimes \mathbb M\to \mathbb M'$, $F(a\otimes m):=f(am)-af(m)$ is the zero morphism. Likewise, $f_R$ is a morphism of $\mathbb A(R)$-modules if and only if $F_R\colon \mathbb A(R)\otimes \mathbb M(R)\to \mathbb M'(R)$, $F_R(a\otimes m)=f_R(am)-af_R(m)$ is the zero morphism. Now, the proposition is a consequence of the injective morphisms,

$$\aligned &\Hom_{\mathcal R}(\mathbb A\otimes \mathbb M,\mathbb M') =
\Hom_{\mathcal R}(\mathbb A,\mathbb Hom_{\mathcal R}(\mathbb M,\mathbb M'))
\underset{\text{\ref{3.10X}}}\hookrightarrow \Hom_{R}(\mathbb A(R),\Hom_{\mathcal R}(\mathbb M,\mathbb M'))
\\ & \underset{\text{\ref{3.10X}}}\hookrightarrow  \Hom_{R}(\mathbb A(R),\Hom_{R}(\mathbb M(R),\mathbb M'(R)))=
  \Hom_{R}(\mathbb A(R)\otimes \mathbb M(R),\mathbb M'(R)).\endaligned$$

\end{proof}

\begin{proposition} \label{5.14b} Let $\mathbb A\in\mathfrak F$ be a functor of  $\mathcal R$-algebras
and let $\mathcal M,\mathcal M'$ be functors of $\mathbb A$-modules. 
Then, $$\Hom_{\mathbb A}(\mathcal M,\mathcal M')=\Hom_{\mathbb A(R)}(M,M').$$ \end{proposition}

\begin{proof} Proceed as in  the proof of Proposition \ref{5.14}.\end{proof}

\begin{notation} Let $M$ be an $R$-module and let $M'\subseteq M$ be an $R$-submodule. By abuse of notation we will say that $\mathcal M'$ is a quasi-coherent submodule of $\mathcal M$.\end{notation}

\begin{proposition} \label{invqua4}
Let $\mathbb A\in\mathfrak F$ be a  functor of ${\mathcal R}$-algebras, let $\mathcal M$ be an $\mathbb A$-module and let $M' \subset M$ be an $R$-submodule.  Then, $\mathcal M'$ is a quasi-coherent
$\mathbb A$-submodule of $\mathcal M$ if and only if $M'$ is an
$\mathbb A(R)$-submodule of $M$.
\end{proposition}

\begin{proof}
Obviously, if $\mathcal M'$ is an $\mathbb A$-submodule of $\mathcal M$ then $M'$ is
an $\mathbb A(R)$-submodule of $M$. Conversely, let us assume  $M'$ is an $
\mathbb A(R)$-submodule of $M$ and let us consider the natural morphism of
multiplication $\mathbb A \otimes_{\mathcal R} {\mathcal M}' \to {\mathcal M}$. By  Lemma \ref{invqua}, the morphisms $\mathbb A\to \mathcal M$, $a\mapsto a\cdot m'$, for each $m'\in M'$, factors uniquely via $\mathcal M'$. Let $\oplus_IR\overset q\to\oplus_JR\overset p\to M'\to 0$ be an exact sequence. Let $i$ be the morphism $\mathcal M'\to \mathcal M$. There exists a (unique) morphism $f'$ such that the diagram
$$\xymatrix{ \mathbb A\otimes_{\mathcal R} (\oplus_I \mathcal R)
\ar[r]^-{Id\otimes q} & \mathbb A\otimes_{\mathcal R} (\oplus_J \mathcal  R) \ar[r]^-{Id\otimes p} \ar[d]^-{f'} & \mathbb A\otimes_{\mathcal R} \mathcal M' \ar[r] \ar[d] & 0\\ & \mathcal M'\ar[r]_-i & \mathcal M & }$$
is commutative. As $i_R\circ f'_R\circ (Id\otimes q)_R=0$, then $f'_R\circ (Id\otimes q)_R=0$. By Proposition \ref{3.10X}, $f'\circ (Id\otimes q)=0$.
Hence, $\mathbb A \otimes_{\mathcal R} {\mathcal M}'\to {\mathcal M}$ factors through $\mathcal M'$.

$F\colon \mathbb A\otimes_{\mathcal R}\mathbb A\to \mathcal M'$, $F(a\otimes a'):=
a(a'm')-(aa')m'$ (for any $m'\in\mathcal M'$) is the zero morphism:
$F$ lifts to a (unique) morphism $\bar F\colon (\mathbb A\otimes_{\mathcal R}\mathbb A)^{**}\to\mathcal M'$. Observe that $i\circ\bar F=0$ because $i\circ F=0$, then $\bar F_R=0$ because
$i_R$ is injective. Finally, $\bar F=0$ because it is determined by $\bar F_R$; and $F=0$.
Likewise, $1\cdot m'=m'$, for all $m'\in\mathcal M'$.

In conclusion, $\mathcal M'$
is a quasi-coherent ${\mathbb A}$-submodule of $\mathcal M$.
\end{proof}

\begin{theorem} \label{3.18}
Let $\mathbb G$ be a functor of monoids and let
$\mathbb M,\mathbb M'$ be  functors of $\mathbb G$-modules.
Assume that $\mathbb A_{\mathbb G}, \mathbb M,\mathbb M'\in\mathfrak F$. 
Let $f\colon \mathbb M\to \mathbb M'$ be a morphism of $\mathcal R$-modules. 
Then, $f$ is a morphism of $\mathbb G$-modules if and only if $f_R$  is a morphism of $\mathbb A_{\mathbb G}^*(R)$-modules.

Let  $\mathcal M$ be a $\mathbb G$-module, then the set of quasi-coherent $\mathbb G$-submodules of $\mathcal M$ is equal to the set of $\mathbb A_{\mathbb G}^*(R)$-submodules of $M$.

\end{theorem}

\begin{proof} It a consequence of Theorem \ref{2.5}, Proposition \ref{5.14} and Proposition \ref{invqua4}.
\end{proof}

\begin{remark} Let $R=K$ be a field. Now assume $\mathfrak F$ is the family of reflexive functors of $\mathcal K$-modules. Recall Remark \ref{Remark}.
All the propositions in this section are likewise true. 
If $M'$ is a $K$-vector subspace, then $M'=\oplus_JK$. This fact simplifies the proof of Proposition \ref{invqua4}.

\end{remark}

\section{Functor of coalgebras}

\begin{notation} \label{notation4.7} Let $\mathbb M_1,\ldots,\mathbb M_n$ be $\mathcal R$-modules, we denote $$\mathbb M_1\tilde\otimes_{\mathcal R}\cdots\tilde\otimes_{\mathcal R}\mathbb M_n:=(\mathbb M_1^*\otimes_{\mathcal R}\cdots\otimes_{\mathcal R}\mathbb M_n^*)^{*}.$$
\end{notation}
Morphisms $\mathbb M_i\to\mathbb N_i$ induce an obvious morphism  $\mathbb M_1\tilde\otimes\cdots\tilde\otimes\mathbb M_n\to \mathbb N_1\tilde\otimes\cdots\tilde\otimes\mathbb N_n$.
Reader can check
\begin{enumerate}
\item $\mathcal M_1\tilde\otimes_{\mathcal R}\cdots\tilde\otimes_{\mathcal R}\mathcal M_n=\mathcal M_1\otimes_{\mathcal R}\cdots\otimes_{\mathcal R}\mathcal M_n$.

\item $\mathcal M_1^*\tilde\otimes_{\mathcal R}\cdots\tilde\otimes_{\mathcal R}\mathcal M_n^*=(\mathcal M_1\otimes_{\mathcal R}\cdots\otimes_{\mathcal R}\mathcal M_n)^*$.

\item  If $\mathbb M_1\tilde\otimes_{\mathcal R}\cdots\tilde\otimes_{\mathcal R}\mathbb M_n$ and 
$\mathbb N_{1}\tilde\otimes_{\mathcal R}\overset{m-n}\cdots\tilde\otimes\mathbb N_m$ are reflexive, then the obvious morphism  $ (\mathbb M_1\tilde\otimes \cdots\tilde\otimes\mathbb M_n)\tilde\otimes(\mathbb N_{1}\tilde\otimes\cdots\tilde\otimes\mathbb N_m)\to 
\mathbb M_1\tilde\otimes\cdots\tilde\otimes\mathbb M_n\tilde\otimes \mathbb N_1\tilde\otimes\cdots\tilde\otimes\mathbb N_m$ is an isomorphism:

$$\aligned (\mathbb M_1\tilde\otimes &\cdots\tilde\otimes\mathbb M_n)\tilde\otimes(\mathbb N_{1}\tilde\otimes\cdots\tilde\otimes\mathbb N_m)\!=\!
\mathbb Hom_{\mathcal R}((\mathbb M_1\tilde\otimes\cdots\tilde\otimes\mathbb M_n)^*\!,(\mathbb N_{1}\tilde\otimes\cdots\tilde\otimes\mathbb N_m)^{**}) \\& =
\mathbb Hom_{\mathcal R}((\mathbb M_1\tilde\otimes\cdots\tilde\otimes\mathbb M_n)^*,\mathbb N_{1}\tilde\otimes\cdots\tilde\otimes\mathbb N_m)\\& =
\mathbb Hom_{\mathcal R}((\mathbb M_1\tilde\otimes\cdots\tilde\otimes\mathbb M_n)^*\otimes\mathbb N_{1}^*\otimes\cdots\otimes\mathbb N_m^*,\mathcal R)\\ &=\mathbb Hom_{\mathcal R}(\mathbb N_{1}^*\otimes\cdots\otimes\mathbb N_m^*,\mathbb M_{1}\tilde\otimes\cdots\tilde\otimes\mathbb M_n)=
\mathbb M_1\tilde\otimes\cdots\tilde\otimes\mathbb M_n\tilde\otimes \mathbb N_1\tilde\otimes\cdots\tilde\otimes\mathbb N_m\endaligned$$
\end{enumerate}

\begin{definition} $\mathbb C$ is said to be a functor of coalgebras if there exist a morphism of $\mathcal R$-modules $m'\colon \mathbb C\to \mathbb C\tilde \otimes_{\mathcal R}\mathbb C$ coassociative (that is, $(m'\tilde \otimes Id)\circ m'=(Id\tilde \otimes m')\circ m'$, where $m'\tilde \otimes Id\colon \mathbb C\tilde \otimes\mathbb C\to
\mathbb C\tilde \otimes\mathbb C\tilde\otimes \mathbb C$ and $Id\tilde \otimes m'$ are the obvious morphisms)
and a counit
(that is, a morphism of $\mathcal R$-modules $u'\colon\mathbb C\to \mathcal R$
such that $(u'\tilde\otimes Id)\circ m'=Id=(Id\tilde\otimes u')\circ m'$, where
$u'\tilde\otimes Id\colon \mathbb C\tilde\otimes_{\mathcal R} \mathbb C\to \mathcal R\tilde\otimes_{\mathcal R}\mathbb C=\mathbb C$ and $Id\tilde\otimes u'$ are the obvious morphisms).

\end{definition}

Let $C$ be an $R$-module.
$C$ is an  $R$-coalgebra if and only if $\mathcal C$ is a functor of  coalgebras. Although we have called functor of coalgebras to $\mathbb C$,
we warn the reader that $\mathbb C(S)$ is not necessarily an $S$-coalgebra

\begin{proposition}  \label{notation4.8}  $\mathbb C$ is  a functor of coalgebras if and only if
$\mathbb C^*$ is a functor of $\mathcal R$-algebras.
In particular, an $R$-module $C$ is a coalgebra if and only if $\mathcal C^*$ is a functor of $\mathcal R$-algebras.
\end{proposition}

\begin{proof} Observe that
$$\Hom_{\mathcal R}(\mathbb C^*\otimes\overset n\cdots\otimes \mathbb C^*,\mathbb C^*)=
\Hom_{\mathcal R}(\mathbb C,(\mathbb C^*\otimes_{\mathcal R}\overset n\cdots\otimes_{\mathcal R} \mathbb C^*)^*)
=\Hom_{\mathcal R}(\mathbb C,\mathbb C\tilde\otimes_{\mathcal R}\overset n\cdots\tilde\otimes_{\mathcal R} \mathbb C).$$
Now, it is easy to complete the proof.
\end{proof}

\begin{notation} \label{notation4.8} Let $\mathbb M_1,\ldots,\mathbb M_n$ be $\mathcal R$-modules, we denote $$\mathbb M_1\bar\otimes_{\mathcal R}\cdots\bar\otimes_{\mathcal R}\mathbb M_n:=(\mathbb M_1\otimes_{\mathcal R}\cdots\otimes_{\mathcal R}\mathbb M_n)^{**}.$$
\end{notation}

Observe that $$(\mathbb M_1\tilde\otimes_{\mathcal R}\cdots\tilde\otimes_{\mathcal R}\mathbb M_n)^*=\mathbb M_1^*\bar\otimes_{\mathcal R}\cdots\bar\otimes_{\mathcal R}\mathbb M_n^*.$$ If $\mathbb M_i$ is reflexive for all $i$, and $\mathbb M_1^*\tilde\otimes_{\mathcal R}\cdots\tilde\otimes_{\mathcal R}\mathbb M_n^*$ is reflexive then
$$(\mathbb M_1\bar\otimes_{\mathcal R}\cdots\bar\otimes_{\mathcal R}\mathbb M_n)^*=
\mathbb M_1^*\tilde\otimes_{\mathcal R}\cdots\tilde\otimes_{\mathcal R}\mathbb M_n^*.$$

If $\mathbb M_1\bar\otimes_{\mathcal R}\cdots\bar\otimes_{\mathcal R}\mathbb M_n$ and $\mathbb N_1\bar\otimes_{\mathcal R}\cdots\bar\otimes_{\mathcal R}\mathbb N_m$ are reflexive, the natural morphism

$$\mathbb M_1\bar\otimes_{\mathcal R}\cdots\bar\otimes_{\mathcal R}\mathbb M_n\bar\otimes \mathbb N_1\bar\otimes_{\mathcal R}\cdots\bar\otimes_{\mathcal R}\mathbb N_m\to
(\mathbb M_1\bar\otimes_{\mathcal R}\cdots\bar\otimes_{\mathcal R}\mathbb M_n)\bar\otimes(\mathbb N_1\bar\otimes_{\mathcal R}\cdots\bar\otimes_{\mathcal R}\mathbb N_m)$$ is an isomorphism (use Proposition \ref{3.2}).

\begin{proposition} \label{harto} Let $\mathbb M_1,\ldots,\mathbb M_n$ be essentially free $\mathcal R$-modules, and let us consider the natural morphism $\mathbb M_1^*\otimes_{\mathcal R}\cdots\otimes_{\mathcal R}\mathbb M_n^*\to (\mathbb M_1\otimes_{\mathcal R}\cdots\otimes_{\mathcal R}\mathbb M_n)^*$. Then, the dual morphism

$$\mathbb M_1\bar\otimes_{\mathcal R}\cdots\bar\otimes_{\mathcal R}\mathbb M_n\to \mathbb M_1\tilde\otimes_{\mathcal R}\cdots\tilde\otimes_{\mathcal R}\mathbb M_n$$ is injective.\end{proposition}

\begin{proof} Write $\mathbb M_j=\plim{\alpha\in P_j} \mathcal \mathcal R^\alpha\subseteq \mathcal R^{I_j}$, where
$P_j$ is a subset of the set of parts of $I_j$.  $(\mathbb M_1\otimes\cdots\otimes\mathbb M_n)^{**}$ and $(\mathbb M_1^*\otimes\cdots\otimes\mathbb M_n^*)^*$ are functors of $\mathcal R$-submodules of $\mathcal R^{\prod_{j\in J} I_j}$, by Theorem \ref{teoremon}  and Theorem \ref{teoremonc}.\end{proof}

\section{Functors of pro-quasicoherent  algebras}

\begin{definition} Let $A$ be an $R$-algebra. The associated functor $\mathcal A$ is obviously a functor of $\mathcal R$-algebras. We will say that $\mathcal A$ is a quasi-coherent $\mathcal R$-algebra.\end{definition}

\begin{definition} \label{proquasi} We will say that a functor of $\mathcal R$-algebras  is a functor of pro-quasicoherent   algebras if it is the inverse limit of its quasi-coherent algebra quotients.\end{definition}

\begin{examples} \label{ejemplor} Quasi-coherent algebras are pro-quasicoherent functors of algebras.

Let $R=K$ be a field, $A$ be a commutative $K$-algebra and $I\subseteq A$ be an ideal. Then, $\mathbb B=\plim{n\in\mathbb N}\mathcal A/\mathcal I^n\in\mathfrak F$  and it is a pro-quasicoherent algebra: $\mathbb B\simeq
\prod_n \mathcal I^n/\mathcal I^{n+1}$, then $\mathbb B\mathfrak F$. $\mathbb B^*=\oplus_n (\mathcal I^n/\mathcal I^{n+1})^*=\ilim{n} (\mathcal A/\mathcal I^n)^*$. Therefore, $\mathbb B^*$ is equal to the direct limit of the dual of the quasi-coherent algebra quotients of $\mathbb B$. Dually, $\mathbb B$ is a pro-quasicoherent algebra.

\end{examples}

\begin{proposition} \label{invqua2} Let $\mathbb A\in\mathfrak F$ and $\mathbb B$ be functors of $\mathcal R$-algebras and assume that there exists an injective morphism of $\mathcal R$-modules $\mathbb B\hookrightarrow \mathcal N$. Then, any morphism of $\mathcal R$-algebras $\phi\colon \mathbb A\to \mathbb B$  factors uniquely through an epimorphism of algebras onto the quasi-coherent algebra associated with $\Ima \phi_R$.\end{proposition}

\begin{proof} By Lemma \ref{invqua}, the morphism $\phi\colon \mathbb A\to \mathbb B$   factors uniquely through
an epimorphism $\phi'\colon \mathbb A\to \mathcal B'$, where $B':=\Ima\phi_R$. Obviously $B'$ is an $R$-subalgebra of $\mathbb B(R)$. We have to check that $\phi'$ is a morphism of functors of algebras.

Observe that if a morphism $f\colon \mathbb A\otimes \mathbb A\to \mathcal N$ factors  through an epimorphism onto a quasi-coherent  submodule $\mathcal N'$ of $\mathcal N$ then   factors  uniquely through $\mathcal N'$, because $f$ and any morphism to $\mathcal N'$   factors uniquely through
$(\mathbb A\otimes \mathbb A)^{**}\in\mathfrak F$.

Consider the diagram

$$\xymatrix{\mathbb A\otimes\mathbb A \ar[d]^-{m_{\mathbb A}} \ar[r]^-{\phi'\otimes\phi'}
& \mathcal B'\otimes\mathcal B' \ar[d]^-{m_{\mathcal B'}}\ar[r]^-{i\otimes i}
& \mathbb B\otimes\mathbb B \ar[d]^-{m_{\mathbb B}} & \\ \mathbb A \ar[r]^-{\phi'} & \mathcal B' \ar[r]^-i & \mathbb B \ar@{^{(}->}[r] & \mathcal N,}$$
where $m_{\mathbb A}, m_{\mathcal B'}$ and $m_{\mathbb B}$ are the multiplication morphisms and $i$ is the morphism induced by the morphism $B'\to \mathbb B(R)$.
We know $m_{\mathbb B}\circ (i\otimes i)\circ (\phi'\otimes \phi')=i\circ \phi'\circ m_{\mathbb A}$. The morphism $m_{\mathbb B}\circ (i\otimes i)\circ (\phi'\otimes \phi')$  factors uniquely  onto $\mathcal B'$, more concretely, through $m_{\mathcal B'}\circ (\phi'\otimes\phi')$.
The morphism $i\circ \phi'\circ m_{\mathbb A}$   factors uniquely onto $\mathcal B'$, effectively, through $\phi'\circ m_{\mathbb A}$. Then, $m_{\mathcal B'}\circ (\phi'\otimes\phi')=\phi'\circ m_{\mathbb A}$ and $\phi'$ is a morphism of $\mathcal R$-algebras.

\end{proof}

\begin{remark} \label{Remark3} Let $R=K$ be a field. Now assume $\mathfrak F$ is the family of reflexive functors of $\mathcal K$-modules. Recall Remark \ref{Remark}. Proposition \ref{invqua2} is likewise true (in this case $\phi'$ is obviously a morphism of functors of algebras).

\end{remark}

\begin{lemma} \label{cua} Let $\mathbb A$ and $\mathbb B$ be two reflexive    functors of pro-quasicoherent   algebras. Let $\{\mathcal A_i\}_i$ and $\{ \mathcal B_j\}$ be the quasi-coherent algebra quotients of $\mathbb A$ and $\mathbb B$. Then,
$$\mathbb A\tilde\otimes_{\mathcal R}\mathbb B=\plim{i,j} (\mathcal A_i\otimes \mathcal B_j).$$
Then, $\mathbb A\tilde\otimes_{\mathcal R}\mathbb B$ is a functor of algebras and
the natural morphism $\mathbb A\otimes \mathbb B\to \mathbb A\tilde\otimes \mathbb B$
is a morphism of functors of algebras.

\end{lemma}

\begin{proof} We have
$$\aligned(\mathbb A^* \otimes \mathbb B^*)^* & =\mathbb Hom_{\mathcal R}(\mathbb A^*,\mathbb B)\overset{\text{\ref{3.2}}}=
\mathbb Hom_{\mathcal R}(\ilim{i} \mathcal A_i^*,\mathbb B)=
\mathbb Hom_{\mathcal R}(\ilim{i} \mathcal A_i^*,\plim{j} \mathcal B_j)\\ & =\plim{i,j}
\mathbb Hom_{\mathcal R}( \mathcal A_i^*,\mathcal B_j)\overset{\text{\ref{prop4}}}=
\plim{i,j} (\mathcal A_i\otimes \mathcal B_j).\endaligned$$
\end{proof} 

\begin{theorem} \label{prodspec} Let $\mathbb A$ and $\mathbb B\in\mathfrak F$ be two pro-quasicoherent   functors of   algebras. Then,
$\mathbb A\tilde\otimes_{\mathcal R}\mathbb B\in \mathfrak F$ is a functor of pro-quasicoherent   algebras and 
$$\Hom_{\mathcal R-alg}(\mathbb A\otimes_{\mathcal R} \mathbb B,\mathbb C)=\Hom_{\mathcal R-alg}\mathbb A\tilde\otimes_{\mathcal R}\mathbb B,\mathbb C)$$
for every functor of pro-quasicoherent   algebras $\mathbb C$.

\end{theorem}

\begin{proof} $\mathbb A\tilde\otimes_{\mathcal R}\mathbb B=(\mathbb A^*\otimes \mathbb B^*)^*\in \mathfrak F$. Let us follow the notations of Lemma \ref{cua}. 

Given a morphism of functor of $\mathcal R$-algebras $\phi\colon \mathbb A\otimes\mathbb B\to \mathcal C$, let $\phi_1=\phi_{|\mathbb A\otimes 1}$ and $\phi_2=\phi_{|1\otimes\mathbb B}$. Then, $\phi_1$ factors through an epimorphism onto  a quasi-coherent algebra quotient $\mathcal A_i$ of $\mathbb A$, and  $\phi_2$ factors through an epimorphism onto a quasi-coherent algebra quotient $\mathcal B_j$ of $\mathbb B$. Then, $\phi$ factors  through $\mathcal A_i\otimes\mathcal B_j$, and $\phi$ factors through $\mathbb A\tilde\otimes_{\mathcal R}\mathbb B$.
Then, $$\Hom_{\mathcal R-alg}(\mathbb A\tilde\otimes_{\mathcal R}\mathbb B,\mathcal C)\to\Hom_{\mathcal R-alg}(\mathbb A\otimes \mathbb B,\mathcal C)$$
is surjective. It is also injective, because
$$\aligned \Hom_{\mathcal R}(\mathbb A\tilde\otimes_{\mathcal R}\mathbb B,\mathcal C)& =\Hom_{\mathcal R}(\mathcal C^*,(\mathbb A^*\otimes \mathbb B^*)^{**})\\ & \overset{\text{\ref{harto}}}\hookrightarrow \Hom_{\mathcal R}(\mathcal C^*,(\mathbb A\otimes \mathbb B)^*)=\Hom_{\mathcal R}(\mathbb A\otimes \mathbb B,\mathcal C).\endaligned$$

Then, $\Hom_{\mathcal R-alg}(\mathbb A\otimes \mathbb B,\mathbb C)=\Hom_{\mathcal R-alg}(\mathbb A\tilde\otimes_{\mathcal R}\mathbb B,\mathbb C)$
for every pro-quasicoherent algebra $\mathbb C$.

A morphism of functors of algebras $f\colon\mathbb A\tilde\otimes_{\mathcal R}\mathbb B\to\mathcal C$ factors through some $\mathcal A_i\otimes\mathcal B_j$ because
$f_{|\mathbb A\otimes\mathbb B}$ factors through some $\mathcal A_i\otimes\mathcal B_j$. Then, the inverse limit of the quasi-coherent algebra quotients of
$\mathbb A\tilde\otimes_{\mathcal R}\mathbb B$ is equal to $\plim{i,j} (\mathcal A_i\otimes\mathcal B_j)=\mathbb A\tilde\otimes_{\mathcal R}\mathbb B$, that is,
$\mathbb A\tilde\otimes_{\mathcal R}\mathbb B$ is a pro-quasicoherent algebra.

\end{proof}

\begin{remark} Moreover,
let $\mathbb A_1,\ldots,\mathbb A_n\in\mathfrak F$ and $\mathbb C$
be pro-quasicoherent   functors of   $\mathcal R$-algebras,  and let
$\phi\in \prod_i\mathbb Hom_{\mathcal R-alg}(\mathbb A_i,\mathbb C)\subseteq \mathbb Hom_{\mathcal R}(\mathbb A_1\otimes_{\mathcal R}\cdots\otimes_{\mathcal R} \mathbb A_n,\mathbb C)$, then
$\phi$  factors uniquely through $\mathbb A_1\tilde\otimes\cdots\tilde\otimes\mathbb A_n$.
\end{remark}

Let $\mathbb A^*\in\mathfrak F$ be a functor of pro-quasicoherent algebras. The multiplication morphism $m\colon \mathbb A^*\otimes \mathbb A^*\to \mathbb A^*$ factors uniquely through
$ \mathbb A^*\tilde\otimes \mathbb A^*$. Then, the comultiplication morphism $\mathbb A\to 
\mathbb A\tilde\otimes \mathbb A$ factors through $\mathbb A\bar\otimes\mathbb A$.
Taking duals in $$\mathbb A^*\tilde\otimes \mathbb A^*\tilde\otimes\mathbb A^*\dosflechasa{m\otimes Id}{Id\otimes m} \mathbb A^*\tilde\otimes \mathbb A^* \overset{m}\to \mathbb A^*,$$ we have that the morphism $\mathbb A\to \mathbb A\bar\otimes\mathbb A$ is coassociative.

\subsection{Functors of procoherent algebras}

\begin{proposition} \label{5.9} Let $\mathcal C^*\in\mathfrak F$ be a functor of $\mathcal R$-algebras. Then, $\mathcal C^*$ is a functor of pro-quasicoherent   algebras.\end{proposition}

\begin{proof} 1. $\mathcal C^*$ is a left and right $\mathcal C^*$-module, then $\mathcal C$ is a right and left $\mathcal C^*$-module. Given $c\in C$, the dual morphism of the morphism
$\mathcal C^*\to \mathcal C$, $w\mapsto w\cdot c$ is the morphism $\mathcal C^*\to \mathcal C$, $w\mapsto c\cdot w$.

2. $C$ is the direct limit of its finitely generated $R$-submodules. Let $N=\langle n_1,\ldots, n_r\rangle\subset C$ be a finitely generated  $R$-module  and let $f\colon \mathcal C^{*r}\to \mathcal N$ be defined by $f((w_i)):=\sum_i w_i\cdot n_i$.
$N':=\Ima f_R$ is a finitely generated $R$-module, by Lemma \ref{x}.
By Proposition \ref{invqua4}, $\mathcal N'$ is a quasi-coherent $\mathcal C^*$-submodule of $\mathcal C$.
Write $N'=\langle n'_1,\ldots, n'_s\rangle$. The morphism $\mathbb End_{\mathcal R}(\mathcal N')\to \oplus^s\mathcal N'$, $g\mapsto (g(n'_i))_i$ is injective. By Proposition \ref{invqua2}, the morphism of functors of $\mathcal R$-algebras $\mathcal C^*\to \mathbb End_{\mathcal R}(\mathcal N')$ $w\mapsto w\cdot$  factors through an epimorphism onto a
quasi-coherent algebra, $\mathcal B'$. The dual morphism of the composite morphism
$$\xymatrix{\mathcal C^* \ar@{->>}[r] & \mathcal B'\ar[r] & \mathbb End_{\mathcal R}(\mathcal N') \ar@{^{(}->}[r] & \oplus^s \mathcal N' \ar[r] &  \oplus^s \mathcal C\ar[r]^-{\pi_i} \ar[r] &\mathcal C\\ w \ar@{|->}[rrrrr] &&&&& w\cdot n'_i}$$
is $\mathcal C^*\to {\mathcal B'}^* \hookrightarrow \mathcal C$, $w\mapsto n'_i\cdot w$.
Hence, $n'_i\in {B'}^*$, for all $i$, and $N'\subseteq B'^*$.  Therefore, $\mathcal C$ is equal to the direct limit of the dual
functors of the quasi-coherent algebra quotients of $\mathcal C^*$. Dually,
$\mathcal C^*$ is a functor of pro-quasicoherent   algebras.

\end{proof}

\begin{observation} \label{5.9b} Recall that if $\mathcal C^*\to \mathcal A$ is an epimorphism,
$A$ is an $R$-module of finite type, by Lemma \ref{x}.\end{observation}

\begin{definition} A functor of $\mathcal R$-algebras $\mathcal C^*\in\mathfrak F$ will be called a   functor of procoherent algebras.\end{definition}

From now on, in this subsection, $R=K$ will be a field. 

\begin{definition} \label{Notation}  Let $\mathbb A$ be a reflexive functor of $\mathcal K$-algebras and let $\{\mathcal A_i\}$ be the set of quasi-coherent algebra quotients of $\mathbb A$ such that
$\dim_K A_i<\infty$. We denote $\bar{\mathbb A}:=\plim{i} \mathcal A_i$. \end{definition}

Let $C_i=A_i^*$ and $C=\ilim{i} A_i^*$,  then  $\bar{\mathbb A}=\plim{i} \mathcal A_i=(\ilim{i} \mathcal A_i^*)^*=\mathcal C^*$ is a  functor of procoherent algebras

Let $\mathcal C^*$ be a functor of $\mathcal K$-algebras, then
$\overline{\mathcal C^*}=\mathcal C^*$ (by Proposition \ref{5.9} and Obs. \ref{5.9b}).

\begin{lemma} \label{N6.7}
Let $\{M_i\}_{i\in I}$ be an inverse system of finite dimensional $K$-vector spaces. Then,
$\mathcal M_i^*$ is a quasicoherent functor of $\mathcal K$-modules and
$$\mathbb Hom_{\mathcal K}(\plim{i\in I}\mathcal M_i,\mathcal N)\overset{\text{\ref{prop4}}}=(\ilim{i\in I} \mathcal M_i^*)\otimes \mathcal N=
\ilim{i\in I} (\mathcal M_i^*\otimes \mathcal N)\overset{\text{\ref{prop4}}}=\ilim{i\in I} \mathbb Hom_{\mathcal K}(\mathcal M_i,\mathcal N).$$

\end{lemma}

\begin{proposition} \label{a5.9} Let $\mathbb A$ be a reflexive functor of $\mathcal K$-algebras. Then,
$$\Hom_{\mathcal K-alg}(\mathbb A,\mathcal C^*)=
\Hom_{\mathcal K-alg}(\bar{\mathbb A},\mathcal C^*),$$
for all  functors of procoherent algebras $\mathcal C^*$.
\end{proposition}

\begin{proof} Let $C$ be a finite dimensional $K$-algebra and let $\{\mathcal A_i\}_{i\in I}$ be the set of all quasi-coherent algebra quotients of $\mathbb A$ such that
$\dim_K A_i<\infty$. Then, $$\aligned \Hom_{\mathcal K-alg}(\mathbb A,\mathcal C)& \overset{\text{\ref{Remark3}}}=
\ilim{i} \Hom_{\mathcal K-alg}(\mathcal A_i,\mathcal C)\overset{\text{\ref{N6.7}}}=
\Hom_{\mathcal K-alg}(\plim{i} \mathcal A_i,\mathcal C)\\ & =\Hom_{\mathcal K-alg}(\bar{\mathbb A},\mathcal C).\endaligned$$
Write $\mathcal C^*=\plim{j} \mathcal C_j$, where $\{C_j\}_{j\in J}$ is
 the set of quasi-coherent algebra quotients of $\mathcal C^*$ such that
$\dim_K C_j<\infty$. Then,

$$\aligned \Hom_{\mathcal K-alg}(\mathbb A,\mathcal C^*) & =\plim{j} \Hom_{\mathcal K-alg}(\mathbb A,\mathcal C_j)=
\plim{j} \Hom_{\mathcal K-alg}(\bar{\mathbb A},\mathcal C_j)\\ & =
\Hom_{\mathcal K-alg}(\bar{\mathbb A},\mathcal C^*).\endaligned$$
\end{proof}

In particular, we have that $\bar{\bar{\mathbb A}}=\mathbb A$ and a morphism of functors of $\mathcal K$-algebras $\bar{\mathbb A}\to \mathcal B$, where $\dim_KB<\infty$, is an epimorphism if and only if the obvious morphism ${\mathbb A}\to \mathcal B$ is an epimorphism.

\begin{corollary} Let $\mathbb A$ be a reflexive functor of $\mathcal K$-algebras.
The category of finite $\mathcal K$-dimensional  $\mathbb A$-modules  is equal to
to the category of finite $\mathcal K$-dimensional  $\bar{\mathbb A}$-modules.
\end{corollary}

\begin{proof} Let $V$ be a finite dimensional $K$-vector space. Then, $End_K(V)$ is a finite dimensional vector space. By Proposition \ref{a5.9} 

$$\Hom_{\mathcal K-alg}(\mathbb A,\mathbb End_{\mathcal K}(\mathcal V))=
\Hom_{\mathcal K-alg}(\bar{\mathbb A},\mathbb End_{\mathcal K}(\mathcal V)).$$
If $\mathcal V$ is an $\mathbb A$-module, then it is naturally an $\bar{\mathbb A}$-module, and reciprocally.

\end{proof}

\begin{proposition} \label{a5.9b}  Let $\mathbb A_1,\ldots,\mathbb A_n$ be reflexive functors of  $\mathcal K$-algebras. Then, $$\overline{\mathbb A_1\otimes_{\mathcal K} \cdots\otimes_{\mathcal K} \mathbb A_n}=\bar{\mathbb A}_1\tilde\otimes_{\mathcal K} \cdots\tilde\otimes_{\mathcal K} \bar{\mathbb A}_n.$$
\end{proposition}

\begin{proof} Let $\{\mathcal A_{ij}\}_{j\in J}$ be the set of all quasi-coherent algebra quotients of $\mathbb A_i$ such that
$\dim_K A_{ij}<\infty$.  Let $B$ be a finite $K$-algebra.  Any morphism
$\mathbb A_1\otimes\cdots\otimes\mathbb A_n\to \mathcal B$ of functors of $\mathcal K$-algebras factors through some morphism $\mathcal  A_{1i_1}\otimes\cdots\otimes\mathcal A_{ni_n}\to \mathcal B$. Likewise,
any morphism
$\bar{\mathbb A}_1\otimes\cdots\otimes\bar{\mathbb A}_n\to \mathcal B$ of functors of $\mathcal K$-algebras factors through some morphism $\mathcal  A_{1i_1}\otimes\cdots\otimes\mathcal A_{ni_n}\to \mathcal B$. Then,
$$\aligned \Hom & _{\mathcal K-alg}(\overline{\mathbb A_1\otimes_{\mathcal K} \cdots\otimes_{\mathcal K} \mathbb A_n},\mathcal C^*)=\Hom_{\mathcal K-alg}({\mathbb A_1\otimes_{\mathcal K} \cdots\otimes_{\mathcal K} \mathbb A_n},\mathcal C^*)\\&=\Hom_{\mathcal K-alg}({\bar{\mathbb A}_1\otimes_{\mathcal K} \cdots\otimes_{\mathcal K} \bar{\mathbb A}_n},\mathcal C^*)=
\Hom_{\mathcal K-alg}(\bar{\mathbb A}_1\tilde\otimes_{\mathcal K} \cdots\tilde\otimes_{\mathcal K} \bar{\mathbb A}_n,\mathcal C^*)\endaligned$$

\end{proof}

\begin{remark} Moreover,
let $\mathbb A_1,\ldots,\mathbb A_n\in\mathfrak F$ and $\mathbb C$
be pro-quasicoherent   functors of   $\mathcal R$-algebras,  and let
$\phi\in \prod_i\mathbb Hom_{\mathcal R-alg}(\mathbb A_i,\mathbb C)\subseteq \mathbb Hom_{\mathcal R}(\mathbb A_1\otimes_{\mathcal R}\cdots\otimes_{\mathcal R} \mathbb A_n,\mathbb C)$, then
$\phi$  factors uniquely through $\mathbb A_1\tilde\otimes\cdots\tilde\otimes\mathbb A_n$.
\end{remark}

\section{Applications to Algebraic Geometry}

\begin{definition} \label{4.15}
Given a functor of commutative $\mathcal R$-algebras $\mathbb {A}$, the functor
${\rm Spec}\, \mathbb {A}$, ``spectrum of $\mathbb {A}$'', is defined
to be $$({\rm Spec}\, \mathbb {A})(S) := {\rm Hom}_{\mathcal R-alg}
(\mathbb {A}, {\mathcal S}),$$ for every commutative $R$-algebra $S$.
\end{definition}

\begin{proposition} \label{4.16}
Let $\mathbb A$ be a functor of commutative $\mathcal R$-algebras. Then, $${\rm
Spec}\,{\mathbb A}={\mathbb Hom}_{\mathcal R-alg}({\mathbb A},{\mathcal R}).$$
\end{proposition}

\begin{proof}
By Adjunction formula (\ref{adj}), restricted to
the morphisms of algebras, we have that
$${\mathbb Hom}_{\mathcal R-alg} ({\mathbb A}, {\mathcal R}) (S)
= {\rm Hom}_{\mathcal S-alg} ({\mathbb A}_{|S}, {\mathcal S})
= {\rm Hom}_{\mathcal R-alg} ({\mathbb A}, {\mathcal S})  = ({\rm Spec}\, {\mathbb A}) (S).$$
\end{proof}

Therefore, ${\rm Spec}\,{\mathbb A}={\mathbb Hom}_{\mathcal R-alg} ({\mathbb A},{\mathcal
R}) \subset {\mathbb Hom}_{\mathcal R}({\mathbb A},{\mathcal R}) = {\mathbb A}^*$.

\begin{proposition} \label{homspe} Let $\mathbb X$ be a functor of sets and let $\mathbb A_{\mathbb X}:=\mathbb Hom(\mathbb X,\mathcal R)$ be its functor of functions. Then,

$${\rm Hom}(\mathbb X,\Spec \mathbb B)={\rm Hom}_{\mathcal R-alg}(\mathbb B,\mathbb A_{\mathbb X}),$$ for every functor of commutative $\mathcal R$-algebras, $\mathbb B$.

\end{proposition}

\begin{proof} Given $f\colon \mathbb X\to \Spec\mathbb B$, let  $f^*\colon \mathbb B\to \mathbb A_{\mathbb X}$ be defined by $f^*(b)(x):=f(x)(b)$, for every $x\in \mathbb X$. Given $\phi\colon \mathbb B\to \mathbb A_{\mathbb X}$, let
$\phi^*\colon \mathbb X\to \Spec\mathbb B$ be defined by $\phi^*(x)(b):=
\phi(b)(x)$, for all $b\in\mathbb B$. It is easy to check that $f=f^{**}$ and $\phi=\phi^{**}$.

\end{proof}

\begin{example} \label{ejemplo}
If $A$ is a commutative $R$-algebra, then ${\rm Spec}\, {\mathcal A} =
({\rm Spec}\, A)^\cdot$ and $\mathbb A_{\Spec\mathcal A}={\mathbb Hom}(\Spec\mathcal A,\mathcal R)={\mathbb Hom}(\Spec\mathcal A,\Spec \mathcal R[x])=\mathbb Hom_{\mathcal R-alg}(\mathcal R[x],\mathcal A)=\mathcal A$.
\end{example}

\begin{definition} We will say that a functor of sets $\mathbb X$ is affine
when $\mathbb X=\Spec\mathbb A_{\mathbb X}$ and $\mathbb A_{\mathbb X}$ is a reflexive functor of $\mathcal R$-modules.
\end{definition}

Let $\mathbb Y$ be an affine functor. By Proposition \ref{homspe},
\begin{equation} \label{EqHom} \Hom(\mathbb X,\mathbb Y)=\Hom_{\mathcal R-alg}(\mathbb A_{\mathbb Y},\mathbb A_{\mathbb X}).\end{equation}

\begin{example} \label{eje4.5} Affine schemes, $\Spec \mathcal A$, are affine functors, by Example
\ref{ejemplo}.\end{example}

\begin{proposition} \label{3.10} We have that
$\mathbb A_{\ilim{i}\mathbb X_i}=\plim{i}\mathbb A_{\mathbb X_i}$.
\end{proposition}

\begin{proof}

$\mathbb A_{\ilim{i}\mathbb X_i} ={\mathbb  Hom}
(\underset{\underset{i}{\longrightarrow}} {\lim}\, \mathbb X_i, \mathcal R)  =
\underset{\underset{i}{\longleftarrow}}{\lim}\,{\mathbb  Hom}
(\mathbb X_i,\mathcal R)
= \plim{i}\mathbb A_{\mathbb X_i}$.
\end{proof}

\begin{theorem} \label{t1.9} \label{t1.92} Let $\mathbb A\in\mathfrak F$ be a functor of commutative algebras.
Let $\{\mathcal A_i\}_i$ be the set of all quasi-coherent algebra quotients of $\mathbb A$. Then,
$$\aligned & 1.\,\, \Spec\mathbb A =   \mathbb Hom_{\mathcal R-alg}(\mathbb A,\mathcal R)\overset{\text{\ref{invqua2}}}=
\ilim{i} \mathbb Hom_{\mathcal R-alg}(\mathcal A_i,\mathcal R)=\ilim{i} \Spec \mathcal A_i.\\ & 2.\,\, \mathbb A_{\Spec\mathbb A}\overset{\text{\ref{3.10}}} = \plim{i}  \mathcal A_i.\endaligned$$
If $\mathbb X$ is an affine functor then
$\mathbb A_{\mathbb X}$ is a pro-quasicoherent algebra.
\end{theorem}

\begin{proof}
Assume $\Spec \mathbb A$ is affine. Then, $\Hom_{\mathcal R-alg}(\mathbb A_{\Spec\mathbb A},\mathcal B)=\Hom_{\mathcal R-alg}(\mathbb A,\mathcal B)$.  
Given a morphism 
$\mathbb A_{\Spec\mathbb A}\to \mathcal B$ of functors of $\mathcal R$-algebras,  the  composite morphism $\mathbb A\to \mathbb A_{\Spec\mathbb A}\to \mathcal B$
 factors through a quotient $\mathcal A_i$, then $\mathbb A_{\Spec\mathbb A}\to \mathcal B$ factors through $\mathcal A_i$ too.
Hence, $\{\mathcal A_i\}$ is the set of all quasi-coherent algebra quotients of $\mathbb A_{\Spec\mathbb A}$. Then,
$\mathbb A_{\Spec\mathbb A}$ is a pro-quasicoherent algebra, by 2.

\end{proof}

\begin{corollary} \label{t1.9b}
The category of commutative pro-quasicoherent algebras, $\mathbb A\in\mathfrak F$,  is anti-equivalent to the category of affine functors, $\mathbb X$, such that $\mathbb A_{\mathbb X} \in\mathfrak F$. The functors $\mathbb A\rightsquigarrow \Spec\mathbb A$,
$\mathbb X\rightsquigarrow \mathbb A_{\mathbb X}$ establish this anti-equivalence.
\end{corollary}

\begin{remark} Let $R=K$ be a field. Now, assume $\mathfrak F$ is the family of reflexive functors of $\mathcal K$-modules. Recall Remark \ref{Remark3}. Theorem \ref{t1.9} is likewise true.
Then, the  category of reflexive functors of commutative pro-quasicoherent algebras,  is anti-equivalent to the category of affine functors.

\end{remark}

\begin{proposition} \label{610} Let $\mathbb X,\mathbb Y$ be functors of sets such that $\mathbb A_{\mathbb X}$ and $\mathbb A_{\mathbb Y}$ are reflexive functors, then
$\mathbb A_{\mathbb X\times \mathbb Y}=\mathbb A_{\mathbb X}\tilde\otimes \mathbb A_{\mathbb Y}$.
\end{proposition}

\begin{proof}
$\mathbb Hom(\mathbb X\times \mathbb Y, \mathcal R)=\mathbb Hom(\mathbb X, \mathbb Hom(\mathbb Y, \mathcal R))=\mathbb Hom(\mathbb X,\mathbb A_{\mathbb Y})\overset{\text{\ref{2.3}}}=
\mathbb Hom_{\mathcal R}(\mathbb A_{\mathbb X}^*,\mathbb A_{\mathbb Y})=(\mathbb A_{\mathbb X}^*\otimes \mathbb A_{\mathbb Y}^*)^*=\mathbb A_{\mathbb X}\tilde\otimes \mathbb A_{\mathbb Y}$.\end{proof}

\begin{proposition}
Let  $\mathbb X,\mathbb Y$ be affine functors such that
$\mathbb A_{\mathbb X}, \mathbb A_{\mathbb Y}\in\mathfrak F$, then
$\mathbb X\times\mathbb Y$ is an affine functor and  $\mathbb
A_{\mathbb X\times \mathbb Y}\in\mathfrak F$.\end{proposition}

\begin{proof} $\mathbb
A_{\mathbb X\times \mathbb Y}\overset{\text{\ref{610}}}=\mathbb A_{\mathbb X}\tilde\otimes \mathbb A_{\mathbb Y}\in\mathfrak F$.
$\Spec \mathbb A_{\mathbb X\times \mathbb Y}=
\Spec (\mathbb A_{\mathbb X}\tilde\otimes \mathbb A_{\mathbb Y})\underset{\text{\ref{prodspec}}}=\Spec (\mathbb A_{\mathbb X}\otimes \mathbb A_{\mathbb Y})=\mathbb X\times \mathbb Y.$
\end{proof}

\subsection{Formal schemes}

\begin{definition} Let $\mathcal C^*\in\mathfrak F$ be a functor of commutative algebras.
We will say that
$\Spec\mathcal C^*$ is a formal scheme.
If $\Spec\mathcal C^*$ is a functor of monoids we will say that it is a formal monoid. \end{definition}

Recall that  $\mathcal C^*\in\mathfrak F$ if and only if $C$ is a free $R$-module.

\begin{note} \label{n3.22} By Proposition \ref{5.9} and Corollary \ref{t1.9b}, formal schemes
are affine functors and $\mathbb A_{\Spec \mathcal C^*}=\mathcal C^*$.
 Besides, $\Spec\mathcal C^*$ is a direct limit of finite $R$-schemes (see \ref{t1.9} and \ref{5.9b}). Reciprocally, if $R$ is a field, a direct limit of finite $R$-schemes is a formal scheme, by Theorem \ref{4.4}.
If $R$ is a field, Demazure (\cite{Demazure2}) defines a formal scheme as a functor (from the category of $R$-finite dimensional rings to sets) which is a direct limit of finite $R$-schemes.
\end{note}

The direct product $\Spec\mathcal C_1^*\times \Spec \mathcal C_2^*=\Spec (\mathcal C_1^*\tilde \otimes\mathcal C_2^*)=\Spec (\mathcal C_1 \otimes\mathcal C_2)^* $ of formal schemes is a formal scheme.

\begin{theorem} \label{8} Let $\Spec \mathcal C^*$ be a formal scheme. Every morphism $\Spec \mathcal C^*\to X=\Spec A$ factors uniquely via $\Spec C^*$, that is,

$$\Hom(\Spec {\mathcal C}^*, X^\cdot)=\Hom_{R-sch}(\Spec {C}^*, X).$$ \end{theorem}

\begin{proof} By Equation \ref{EqHom},
$$\aligned \Hom(\Spec\mathcal C^*,\Spec A) &=\Hom_{\mathcal R-alg}(\mathcal A, \mathcal C^*)=
\Hom_{R-alg}(A, C^*)\\ & =\Hom_{R-sch}(\Spec C^*,\Spec A).\endaligned$$
\end{proof}

 \label{seccion4}

\begin{theorem} \label{4.4} Let $\{{\Spec \mathcal C_i^*}\}_{i\in I}$ be a direct system of formal schemes, or equivalently let $\{\mathcal C_i^*\in\mathfrak F\}$ be an inverse system of functors of algebras. Write $C=\ilim{i\in I} C_i$, then $\mathcal C^*=\plim{i\in I} \mathcal C_i^*$.
We have that $$\ilim{i} \Spec \mathcal C_i^* = \Spec (\plim{i}\mathcal C_i^*)=
\Spec \mathcal C^*$$
and $\mathbb A_{\Spec \mathcal C^*}=\mathcal C^*$. 
\end{theorem}

\begin{proof} Observe that
$$\aligned {\rm Hom}_{\mathcal R}( {\mathcal C^*\otimes \cdots\otimes \mathcal C^*}, {\mathcal S}) & =C\otimes\cdots\otimes C\otimes S=\ilim{i} (C_i\otimes\cdots\otimes C_i\otimes S)\\ &=
\ilim{i} {\rm Hom}_{\mathcal R}( {\mathcal C_i^*\otimes \cdots\otimes \mathcal C_i^*}, {\mathcal S})\endaligned$$
Then the kernel of the morphism $ {\rm Hom}_{\mathcal R}( {\mathcal C^*}, {\mathcal S})\to {\rm Hom}_{\mathcal R}( {\mathcal C^*\otimes \mathcal C^*}, {\mathcal S})$, $f\mapsto \tilde f$, where $\tilde f(c_1\otimes c_2)=f(c_1c_2)-f(c_1)f(c_2)$ coincides with the kernel of the morphism
$\underset{\underset{i\in I}{\longrightarrow}}{\lim}\,
{\rm Hom}_{\mathcal R}( {\mathcal C_i}^*,{\mathcal S})\to
\underset{\underset{i\in I}{\longrightarrow}}{\lim}\,
{\rm Hom}_{\mathcal R}( {\mathcal C_i}^*\otimes {\mathcal C_i}^*,{\mathcal S})$, $(f_i)\mapsto (\tilde f_i)$.
Then, ${\rm Hom}_{\mathcal R-alg}(  {\mathcal C^*}, {\mathcal S})  =
\underset{\underset{i\in I}{\longrightarrow}}{\lim}\,
{\rm Hom}_{\mathcal R-alg}( {\mathcal C_i}^*,{\mathcal S})$ and

$$(\Spec {\mathcal C^*})
 (S)=(\underset{\underset{i\in I}{\longrightarrow}}{\lim}\,\Spec {\mathcal C_i}^*)(S).$$
Finally, $\mathbb A_{\Spec {\mathcal C^*}}=\plim{i} \mathbb A_{\Spec {\mathcal C_i^*}}=
\plim{i} \mathcal C_i^*=\mathcal C^*$.

\end{proof}

From now on, in this section, we will assume that $R=K$ is a field.

\begin{definition} \label{4.1} Let $X$ be a $K$-scheme and let  $I$ be the set of all finite $K$-subschemes of $X$.
Given $K$-scheme $Y$ write $A_Y :=\mathcal O_Y(Y)$, the ring of (regular) functions of $Y$.
Define $\bar {\mathcal A}_X:=\plim{i\in I} \mathcal A_i$ and $$\bar X:=\Spec \bar{\mathcal A}_X\overset{\text{\ref{4.4}}}=\ilim{i\in I} \Spec \mathcal A_i$$ That is, ``$\bar X$ is the direct limit of the set of all finite subschemes of $X$''.
\end{definition}

$\bar X$ is a formal scheme and we have a natural monomorphism $\bar X\hookrightarrow X^\cdot$.

\begin{theorem} \label{universal} Let $X$ be a $K$-scheme. Then: $$\Hom(\Spec {\mathcal C^*}, X^\cdot)=
\Hom(\Spec {\mathcal C^*}, \bar X),$$ for every formal scheme $\Spec\mathcal C^*$.\end{theorem}

\begin{proof} ${\mathcal C}^*=\plim{i} \mathcal S_i$, where the algebras $S_i$ are finite $K$-algebras.
Then,
 $$\begin{array}{l} \Hom(\Spec {\mathcal C^*}, X^\cdot)\overset{\text{\ref{4.4}}}=\Hom(\ilim{i} \Spec {\mathcal S_i}, X^\cdot) =
\plim{i}\Hom( \Spec {\mathcal S_i}, X^\cdot)\\= \plim{i} \Hom(\Spec {\mathcal S_i}, \bar X)= \Hom(\ilim{i} \Spec {\mathcal S_i}, \bar X)=
\Hom(\Spec {\mathcal C^*}, \bar X).\end{array} $$\end{proof}

\section{Functors of bialgebras}

\begin{definition} \label{bialgebras} A reflexive functor  $\mathbb B$ of proquasicoherent\footnote{We have assumed $\mathbb B$ proquasi-coherent in order that $\mathbb B\tilde\otimes\mathbb B$ be a  functor of algebras} algebras 
is said to be a functor of bialgebras if 
$\mathbb B^*$ is a  functor of   $\mathcal R$-algebras and the dual morphisms of the multiplication morphism $m\colon \mathbb B^*\otimes \mathbb B^*\to \mathbb B^*$ and the unit morphism $u\colon \mathcal R\to \mathbb B^*$ are morphisms of functors of $\mathcal R$-algebras.

Let $\mathbb B,\mathbb B'$ be two functors of bialgebras. We will say that a morphism of $\mathcal R$-modules, $f\colon \mathbb B\to\mathbb B'$ is a morphism of functors of bialgebras if $f$ and $f^* \colon {\mathbb B'}^*\to\mathbb B^*$ are morphisms of functors of $\mathcal R$-algebras.

\end{definition}

By Proposition \ref{notation4.8}, we can give the following equivalent definition of functor of bialgebras.

\begin{definition} A reflexive functor  $\mathbb B$ of proquasicoherent algebras 
is said to be a functor of bialgebras if it is a functor of coalgebras and the comultiplication morphism $\mathbb B\to \mathbb B\tilde \otimes_{\mathcal R}\mathbb B$ and de counit morphism $\mathbb B\to \mathcal R$  are morphisms of functors of  $\mathcal R$-algebras.
\end{definition}

In the literature, an $R$-algebra $B$ is said to be a bialgebra if it is a coalgebra (with counit) and the comultiplication  $c \colon B\to B\otimes_R B$ and the counit $e\colon B\to R$ are  morphisms of $R$-algebras.

\begin{proposition} \label{5.24} The functors $B \rightsquigarrow \mathcal B$ and $\mathcal B\rightsquigarrow \mathcal B(R)$ establish an equivalence between the category of $R$-bialgebras and the category of functors of $\mathcal R$-bialgebras.\end{proposition}

\begin{proof} Recall \ref{notation4.7} (1) and \ref{notation4.8}.\end{proof}

 If
$\mathcal B$ and $\mathcal B^*\in\mathfrak F$ are functors of $\mathcal R$-algebras, then they are functors of proquasicoherent algebras,  by Proposition \ref{5.9}.

\begin{definition} A functor $\mathbb B$ of bialgebras is said to be a functor of pro-quasicoherent bialgebras  if $\mathbb B^*$ is a  functor of pro-quasicoherent algebras.\end{definition}

\begin{theorem} \label{dualbial} Let ${\mathcal C}_{\mathfrak F-Bialg.}$ be the category of
functors $\mathbb B\in\mathfrak F$ of pro-quasicoherent bialgebras. The functor ${\mathcal C}_{\mathfrak F-Bialg.}\rightsquigarrow{\mathcal C}_{\mathfrak F-Bialg.}$, $\mathbb B \rightsquigarrow {\mathbb B}^*$ is a categorical anti-equivalence.
\end{theorem}

\begin{proof}  Let $\{\mathbb B,m,u; \mathbb B^*,m',u'\}$  be a functor of pro-quasicoherent bialgebras. Let us only check that $m^*\colon \mathbb B^*\to (\mathbb B\otimes\mathbb B)^*=\mathbb B^*\tilde\otimes\mathbb B^*$ is a morphism of functors of algebras.
By hypothesis, ${m'}^*\colon \mathbb B\to (\mathbb B^*\otimes \mathbb B^*)^*=\mathbb B\tilde\otimes\mathbb B$ is a morphism of functors of algebras.
We have the commutative square:
$$(*) \qquad\qquad \xymatrix{ \mathbb B \ar[rr]^-{{m'}^*} & &\mathbb B\tilde\otimes \mathbb B \\ \mathbb B \otimes \mathbb B \ar[u]^-m \ar[rr]^-{{m'}^*_{13}\otimes {m'}^*_{24}}& & \mathbb B\tilde \otimes \mathbb B\tilde \otimes \mathbb B\tilde \otimes \mathbb B, \ar[u]_-{m\otimes m}} $$
where $({m'}^*_{13}\otimes {m'}^*_{24})(b_1\otimes b_2):=
\sigma({m'}^*(b_1)\otimes {m'}^*(b_2))$ and $\sigma(b_1\otimes b_2\otimes b_3\otimes b_4):=b_1\otimes b_3\otimes b_2\otimes b_4$, for all $b_i\in\mathbb B$.
$m'\colon \mathbb B^*\otimes \mathbb B^*\to \mathbb B^*$ factors through $\mathbb B^*\tilde\otimes \mathbb B^*$, because $\mathbb B^*$ is a pro-quasicoherent algebra. Dually, we have the obvious morphisms
$$\mathbb B\overset{m'^*}\to  \mathbb B\bar\otimes \mathbb B\overset{I\otimes I}\longrightarrow \mathbb B\tilde\otimes \mathbb B$$
and the diagram

$$\xymatrix{ \mathbb B \ar[rr]^-{m'^*} && \mathbb B\bar\otimes \mathbb B \ar[rr]^-{I\otimes I}  && \mathbb B\tilde\otimes \mathbb B \\ \mathbb B \otimes \mathbb B \ar[u]^-m \ar[rr]^-{{m'}^*_{13}\otimes {m'}^*_{24}} && \mathbb B\bar \otimes \mathbb B\bar \otimes \mathbb B\bar \otimes \mathbb B, \ar[u]_-{m\otimes m} \ar[rr]^-{I\otimes I\otimes I\otimes I} &&\mathbb B\tilde \otimes \mathbb B\tilde \otimes \mathbb B\tilde \otimes \mathbb B, \ar[u]_-{m\otimes m}} $$
The right square is commutative and the left square is commutative because the diagram $(*)$ is commutative and $I\otimes I$ is injective (by Proposition \ref{harto}).

Taking duals, in the left square, we obtain the commutative diagram:

$$\xymatrix{ \mathbb B^* \ar[d]_-{m^*} &  &
\mathbb B^*\tilde\otimes \mathbb B^* \ar[ll]^-{m'} \ar[d]^-{(m\otimes m)^*}
\\ \mathbb B^*\tilde\otimes \mathbb B^* &  &
\mathbb B^*\tilde\otimes \mathbb B^*\tilde\otimes  \mathbb B^*\tilde\otimes \mathbb B^*
\ar[ll]^-{m'_{13}\otimes m'_{24}}}$$
which shows that $m^*$ is a morphism of functors of $\mathcal R$-algebras.

\end{proof}

In \cite[Ch. I, \S 2, 13]{dieudonne}, Dieudonné proves the anti-equivalence between
the category of commutative $K$-bialgebras and the  category of
linearly compact  cocommutative $K$-bialgebras (where $K$ is a field).

\begin{remark} Let $R=K$ be a field and let  ${\mathcal C}_{proq-bialg.}$ be the category of
functors  of pro-quasicoherent bialgebras. Likewise, the functor ${\mathcal C}_{proq-Bialg.}\rightsquigarrow{\mathcal C}_{proq-bialg.}$, $\mathbb B \rightsquigarrow {\mathbb B}^*$ is a categorical anti-equivalence.

\end{remark} 

From now on, in this section, $R=K$ will be a field.  Recall Notation \ref{Notation}.

\begin{theorem} \label{T6.9} Let $\mathbb B\in\mathfrak F$ be a functor of   $\mathcal K$-bialgebras. Then, $\bar{\mathbb B}\in \mathfrak F$ is a functor of bialgebras and

$$\Hom_{\mathcal K-bialg}(\mathbb B,\mathcal C^*)=
\Hom_{\mathcal K-bialg}(\bar{\mathbb B},\mathcal C^*),$$
for all functor of  bialgebras  $\mathcal C^*$.
\end{theorem}

\begin{proof} Given any $\mathbb A_1,\cdots,\mathbb A_n\in\mathfrak F$ pro-quasicoherent algebras then $\overline{\mathbb A_1\tilde\otimes\cdots\tilde\otimes \mathbb A_n}=
\bar{\mathbb A_1}\tilde\otimes\cdots\tilde\otimes \bar{\mathbb A_n}$ because
$$\aligned \Hom & _{\mathcal K-alg}(\overline{\mathbb A_1\tilde\otimes\cdots\tilde\otimes \mathbb A_n},\mathcal C^*)\overset{\text{\ref{a5.9} }}=\Hom_{\mathcal K-alg}({\mathbb A_1\tilde\otimes\cdots\tilde\otimes \mathbb A_n},\mathcal C^*)\\ & \overset{\text{\ref{prodspec}}}=\Hom_{\mathcal K-alg}({\mathbb A_1\otimes\cdots\otimes \mathbb A_n},\mathcal C^*)\overset{\text{\ref{a5.9} }}=\Hom_{\mathcal K-alg}(\overline{\mathbb A_1\otimes\cdots\otimes \mathbb A_n},\mathcal C^*)\\ &\overset{\text{\ref{a5.9b} }}=\Hom_{\mathcal K-alg}(\bar{\mathbb A_1}\tilde\otimes\cdots\tilde\otimes \bar{\mathbb A_n},\mathcal C^*)\endaligned$$
Then, the comultiplication morphism $\mathbb B\to \mathbb B\tilde\otimes \mathbb B$ defines a comultiplication morphism $\bar{\mathbb B}\to \bar{\mathbb B}\tilde\otimes \bar{\mathbb B}$, and
$\bar{\mathbb B}$ is a bialgebra scheme.

Given a morphism of functors of bialgebras $f\colon \mathbb B\to \mathcal C^*$, that is, a morphism of functors of algebras such that the diagram
$$\xymatrix{\mathbb B \ar[r] \ar[d]^-f & \mathbb B\tilde\otimes\mathbb B\ar[d]^-{f\otimes f}\\ \mathcal C^* \ar[r] & \mathcal C^*\tilde\otimes\mathcal C^*}$$
is commutative, the induced morphism of functors algebras $\bar{\mathbb B}\to \mathcal C^*$ is a morphism of functors of bialgebras. Reciprocally, given a morphism of bialgebras $\bar{\mathbb B}\to \mathcal C^*$, the composite morphism
$\mathbb B\to \bar{\mathbb B}\to \mathcal C^*$ is a morphism of functors of bialgebras.

\end{proof}

\begin{corollary} \label{cor1.7} Let $A$ and $B$ be  $K$-bialgebras. Then,
$$\Hom_{\mathcal K-bialg}(\bar{\mathcal A},{\mathcal B^*})=\Hom_{\mathcal K-bialg}(\bar {\mathcal B},{\mathcal A^*}).$$
\end{corollary}

\begin{proof} It follows from the equalities
$\Hom_{\mathcal K-bialg}(\bar{\mathcal A},{\mathcal B^*})\overset{\text{\ref{T6.9}}}=\Hom_{\mathcal K-bialg}({\mathcal A},{\mathcal B^*})\overset{\text{\ref{dualbial}}}=\Hom_{\mathcal K-bialg}({\mathcal B},{\mathcal A^*})$ $\overset{\text{\ref{T6.9}}}=\Hom_{\mathcal K-bialg}(\bar {\mathcal B},{\mathcal A^*})$.

\end{proof}

\begin{note} \label{cite{E}} The bialgebra $A^\circ :=\Hom_{\mathcal K}(\bar{\mathcal A},\mathcal K)$ is sometimes known as the ``dual bialgebra" of $A$ and Corollary \ref{cor1.7}  says (dually) that the functor assigning to each bialgebra its dual bialgebra is autoadjoint (see \cite[3.5]{E}).
\end{note}

\begin{proposition} Let $k$ be a field, $L$ a Lie algebra and $A:=U(L)$ the universal enveloping algebra of $L$. Let $D_{\mathcal A}:=\{w\in A^*\colon \dim_k\langle AwA\rangle <\infty\}$
and $G:=\Spec D_{\mathcal A}$.
The category of finite dimensional linear representations of $L$ is equivalent to the category of finite dimensional linear representations of $G$.
\end{proposition}

\begin{proof} The category of finite dimensional linear representations of $L$ is equal to the category of $K$-finite dimensional $A$-modules. $A$ is a bialgebra (cocommutative).
Let $V$ be a finite dimensional vector space. We have that
$$\Hom_{K-alg}(A,End_K(V))=\Hom_{\mathcal K-alg}(\mathcal A,\mathbb End_{\mathcal K}(\mathcal V))=
\Hom_{\mathcal K-alg}({\bar{\mathcal A}},\mathbb End_{\mathcal K}(\mathcal V)),$$
by Proposition \ref{a5.9}. Then, the category of finite dimensional linear representations of $L$ is equal to the category of $\mathcal K$-quasicoherent $\bar{\mathcal A}$-modules, $\mathcal V$, such that $\dim_KV<\infty$.
This last category is equivalent to the category of finite dimensional linear representations of $\Spec{\bar{\mathcal A}}^*$, by \ref{E2.20}.

Finally, ${\bar{\mathcal A}\,}^*$ is the quasi-coherent algebra, $\mathcal D_{\mathcal A}$, associated with $D_{\mathcal A}$:
$A\to A_i$ is a $K$-algebra quotient if and only if $A_i$ is a $K$-vector space quotient and it is a right and left $A$-module.  Then, $A\to A_i$ is a $K$-algebra quotient and $\dim_K A_i<\infty$ if and only if $A_i^*$ is a finite dimensional $K$-vector subspace of $A^*$ and  it is a right and left $A$-submodule. Hence, $D_{\mathcal A}=\ilim{i} A_i^*$ 
 and  $\mathcal D_{\mathcal A}={\bar{\mathcal A}\,}^*$. 

\end{proof}

\subsection{Applications to Algebraic Geometry}

\begin{definition} An affine functor $\mathbb G=\Spec \mathbb A$ is said to be an
affine functor of monoids  if $\mathbb G$ is a functor of monoids.
\end{definition}

Let $\mathbb G$ be an affine functor of monoids.
 Let $$m\colon \mathbb G\times \mathbb G\to \mathbb G$$
 be the multiplication morphism and let $e\in\mathbb G$ the identity element. By Theorem \ref{2.5},  $\mathbb A_{\mathbb G}^*$ is a functor of $\mathcal R$-algebras
and  we have a commutative diagram
$$\xymatrix{\mathbb G\times \mathbb G \ar[r]^-m \ar[d]& \mathbb G \ar[d]\\ \mathbb A^*_{\mathbb G}\otimes \mathbb A^*_{\mathbb G} \ar[r]^-m & \mathbb A^*_{\mathbb G}}$$The dual morphisms of the multiplication morphism $m\colon \mathbb A^*_{\mathbb G}\otimes \mathbb A^*_{\mathbb G}\to \mathbb A^*_{\mathbb G}$ and the unit morphism $\mathcal R\to \mathbb A^*_{\mathbb G} $ are
the natural morphisms $\mathbb A_{\mathbb G}\to \mathbb A_{\mathbb G\times \mathbb G}$ and
$\mathbb A_{\mathbb G}\overset e\to \mathcal R$,  $f\mapsto f(e)$, which
are morphisms of $\mathcal R$-algebras. That is, $\mathbb A_{\mathbb X}$ is a functor of bialgebras.

Let $\mathbb X$ be an affine functor and assume $\mathbb A_{\mathbb X}$ is a functor of bialgebras. 
Let $m\colon \mathbb A_{\mathbb X}^*\otimes
\mathbb A_{\mathbb X}^*\to \mathbb A_{\mathbb X}^*$ and  $e\colon \mathcal R\to \mathbb A_{\mathbb X}^*$ the multiplication and unit morphisms. Given a point
$(x,x')\in \mathbb X\times \mathbb X\subset \mathbb Hom_{\mathcal R-alg}(\mathbb A_{\mathbb X\times\mathbb X},\mathcal R)$ then $(x,x')\circ m^*\in \mathbb Hom_{\mathcal R-alg}(\mathbb A_{\mathbb X},\mathcal R)=\mathbb X$ and we have the commutative diagram
$$\xymatrix{\mathbb X\times \mathbb X \ar@{-->}[r]^-m \ar[d]& \mathbb X \ar[d]\\ \mathbb A^*_{\mathbb X}\otimes \mathbb A^*_{\mathbb X} \ar[r]^-m & \mathbb A^*_{\mathbb X}}$$
Obviously $e\in \mathbb Hom_{\mathcal R-alg}(\mathbb A_{\mathbb X},\mathcal R)=\mathbb X$.
It is easy to check that $\{\mathbb X,m,e\}$ is a functor of monoids.

\begin{proposition} Let $\mathbb G$ and $\mathbb G'$ be affine functors of monoids. Then,
$$\Hom_{mon}(\mathbb G,\mathbb G')= \Hom_{\mathcal R-bialg}(\mathbb A_{\mathbb G'},\mathbb A_{\mathbb G}).$$
\end{proposition}

\begin{proof} 
Let $h\colon \mathbb G\to \mathbb G'$ be a morphism of functors of monoids. The composition
morphism of $h$ with the natural morphism $\mathbb G'\to \mathbb A_{\mathbb G'}^*$ factors through $\mathbb A_{\mathbb G}^*$, that is, we have a commutative diagram
$$\xymatrix{\mathbb G\ar[r]^-h \ar[d]& \mathbb G' \ar[d]\\ \mathbb A^*_{\mathbb G}\ar[r] & \mathbb A^*_{\mathbb G'}}$$
The dual morphism $\mathbb A_{\mathbb G'}\to \mathbb A_{\mathbb G}$ is the morphism induced by $h$ between the functors of functions. Conversely, let $f\colon \mathbb A_{\mathbb G'}\to \mathbb A_{\mathbb G}$ be a morphism of functors of $\mathcal R$-algebras, such that $f^*$ is also a morphism  of functors of $\mathcal R$-algebras. Given $g\in\mathbb G$, then $f^*(g)=g\circ f\in
\mathbb Hom_{\mathcal R-alg}(\mathbb A_{\mathbb G'},\mathcal R)=\mathbb G'$.
Hence, $f^*_{|\mathbb G}\colon \mathbb G\to \mathbb G'$ is a morphism of functors of monoids. 
\end{proof}

\begin{theorem} \label{5.3n}  The category of affine functors of  monoids  is anti-equivalent to the category of functors of commutative bialgebras.
\end{theorem}

\begin{theorem} \label{5.3m} The category of cocommutative bialgebras $A$ is equivalent to the category of formal monoids $\Spec \mathcal A^*$ (we assume the $R$-modules $A$ are
free).\end{theorem}

In \cite[Ch. I, \S 2, 14]{dieudonne}, it is given the Cartier Duality (formal schemes are certain functors over the category of commutative linearly compact algebras over a field).

\begin{definition} Let $\mathbb G$ be a functor of abelian monoids.
$\mathbb G^\vee := {\mathbb Hom}_{mon} (\mathbb G,{\mathcal R})$ (where we regard ${\mathcal R}$
as a monoid with the operation of multiplication) is said to be the dual monoid of $\mathbb G$.
\end{definition}

If $\mathbb G$ is a functor of groups, then $\mathbb G^\vee = {\mathbb Hom}_{grp} (\mathbb G, G_m^\cdot)$ ($G_m:=\Spec R[x,1/x]$).

\begin{theorem}\label{dual}
Let  ${\mathbb  G}$ be a functor of abelian
monoids with a reflexive functor of functions. Then,
${\mathbb  G}^\vee={\rm Spec}\, (\mathbb  A_{\mathbb G}^*)$
(in particular, this equality shows that ${\rm Spec}\,\mathbb
A_{\mathbb G}^*$ is a functor of abelian monoids).
\end{theorem}

\begin{proof}
$  \mathbb G^\vee= {\mathbb Hom}_{mon} (
{\mathbb G},{\mathcal R})\overset{\text{\ref{2.5}}}={\mathbb Hom}_{\mathcal R-alg}
(\mathbb  A_{\mathbb G}^*,{\mathcal R})={\rm Spec}\, (\mathbb  A_{\mathbb G}^*).$
\end{proof}

\begin{theorem} \label{Cartier} The category of abelian affine $R$-monoid schemes  $G={\rm Spec}\, A$ is
anti-equiva\-lent to the category of abelian formal monoids ${\rm Spec}\, {\mathcal
A}^*$ (we assume the $R$-modules $A$ are
free). The functor $\mathbb G \rightsquigarrow   \mathbb G^{\vee}$ gives the categorical anti-equivalence.

\end{theorem}


Let $X$ and $Y$ be two $K$-schemes. By the universal property \ref{universal}, it can be checked that
$$\overline{X\times Y}=\bar X\times \bar Y.$$

\begin{theorem} \label{4.4b} Let $G$ be a $K$-scheme on groups (resp. monoids). Then $\bar G$ is a functor of groups (resp. monoids), the natural morphism $\bar G\to
G$ is a morphism of functors of monoids and
$$\Hom_{mon}(\Spec {\mathcal C^*}, G)=
\Hom_{mon}(\Spec {\mathcal C^*}, \bar G),$$ for every formal  monoid $\Spec\mathcal C^*$. If $G$ is commutative, then $\bar G$ is commutative.\end{theorem}

\begin{proof} Let $\mu\colon G\times G\to G$ the multiplication morphism.
By Theorem \ref{universal}, the composite morphism $\overline{G\times G}=\bar G\times \bar G\to G\times G\to G$ factors through a unique morphism $\mu'\colon \bar G\times \bar G\to \bar G$, that is, we have the commutative diagram:
$$\xymatrix{\bar G\times \bar G \ar[r] \ar[d]^-{\mu'} &
G \times G \ar[d]^-\mu \\ \bar G \ar[r] & G}$$
Let $*\colon G\to G$ be the inverse morphism. The composition
$\bar G \to G \overset{*}\to G$ factors through a unique morphism
$*'\colon \bar G\to  \bar G$, that is, we have the commutative diagram:
$$\xymatrix{ \bar G \ar[r] \ar[d]^-{*'} & G \ar[d]^-{*} \\
\bar G\ar[r] & G}$$
Now it is easy to check that $(\bar G,\mu',*')$ is a functor of groups and to conclude the proof.\end{proof}

\begin{proposition} \label{3.8} Let $\Spec \mathcal C^*$ be a formal monoid and $D_{\mathcal C}=\{w\in C^*\colon w(I)=0$ for some bilateral ideal $I\subset C$ of finite codimension$\}\subset C^*$.
Then, $\Spec D_{\mathcal C}$ is an affine monoid scheme and
$$\Hom_{mon}(\Spec \mathcal C^*, \Spec A)=
\Hom_{mon}( \Spec D_{\mathcal C},\Spec A),$$ for every affine monoid scheme $\Spec A$.
\end{proposition}

\begin{proof} Observe that $D_{\mathcal C}=\ilim{I} (C/I)^*$ and $\mathcal D_{\mathcal C}^*=\bar{\mathcal C}$. Then,
$$\aligned \Hom_{mon} & (\Spec \mathcal C^*, \Spec A) \overset{\text{\ref{5.3n}}}=\Hom_{bialg}(\mathcal A,\mathcal C^*)\,\overset{\text{\ref{dualbial}}}=\,\Hom_{bialg}(\mathcal C, {\mathcal A}^*)\\ &\overset{\text{\ref{a5.9}}} =\,
\Hom_{bialg}(\bar{\mathcal C},\mathcal A^*) \overset{\text{\ref{dualbial}}}=\Hom_{bialg}({\mathcal A},\mathcal D_{\mathcal C} ) \overset{\text{\ref{5.3n}}}=
\Hom_{mon}( \Spec D_{\mathcal C},\Spec A).\endaligned$$

\end{proof}

\begin{note}  Let $X$ be a $K$-scheme and $A_X$ the ring of functions of $X$.
The set $D_X$ of distributions of $X$ of finite support is said to be
$D_X:=\{w\in A^*_X\colon w$ factors through a finite quotient algebra of $A_X\}$. Obviously, $\mathcal D_X^*=\bar{\mathcal A}_X$ and $\Spec
\mathcal D_X^*=\bar X$.\end{note}

If $\Spec\mathcal C^*$ is an abelian formal monoid then $G=\Spec C$ is an affine abelian monoid scheme and $D_{\mathcal C}=D_G$, then

\begin{equation} \label{1.9} \Hom_{mon}(G^\vee, \Spec A)=
\Hom_{mon}( \Spec D_{G},\Spec A),\end{equation}  for every affine monoid scheme $\Spec A$.

Assume $G=\Spec A$ and $G'=\Spec B$ are commutative affine monoid schemes, then

\begin{equation} \label{1.92} \aligned
\Hom_{mon}( \Spec D_{G},G') & \overset{\text{Eq.\ref{1.9}}}=\Hom_{mon}(G^\vee, G') =
\Hom_{mon}(G'^\vee, G)\\ & \overset{\text{Eq.\ref{1.9}}}=\Hom_{mon}( \Spec D_{G'},G).\endaligned\end{equation}

\end{document}